\pdfoutput=1
\documentclass{amsart}
\usepackage[T1]{fontenc}   
\usepackage{graphicx}                  

\newtheorem{theorem}{Theorem}[section]
\newtheorem{lemma}[theorem]{Lemma}
\newtheorem{proposition}[theorem]{Proposition}

\newtheorem{conjecture}[theorem]{Conjecture}

\theoremstyle{definition}
\newtheorem{definition}[theorem]{Definition}
\newtheorem{notation}[theorem]{Notation}

\newtheorem{remark}[theorem]{Remark}
\numberwithin{equation}{section}

\usepackage{amsthm}
\usepackage{amssymb}
\usepackage{tabularx}

\begin{document}

\noindent                                             
\begin{picture}(150,36)                               
\put(5,20){\tiny{Submitted to}}                       
\put(5,7){\textbf{Topology Proceedings}}              
\put(0,0){\framebox(140,34){}}                        
\put(2,2){\framebox(136,30){}}                        
\end{picture}                                        
\vspace{0.5in}

\renewcommand{\sc}{\scshape}
\vspace{0.5in}

\begin{abstract}
This paper provides the complete table of prime knot projections with their mirror images, without redundancy,  up to eight double points systematically thorough a finite procedure by flypes.  
In this paper, we show how to tabulate the knot projections up to eight double points by listing tangles with at most four double points by an approach with respect to rational tangles of J.~H.~Conway.  In other words, for a given prime knot projection of an alternating knot, we show how to enumerate possible projections of the alternating knot. 
Also to tabulate knot projections up to ambient isotopy, we introduce arrow diagrams (oriented Gauss diagrams) of knot projections having no over/under information of each crossing, which were originally
introduced as arrow diagrams of knot diagrams by M.~Polyak and O.~Viro.  Each arrow diagram of a knot projection completely detects the difference between the knot projection and its mirror image.   
\end{abstract}       

\author{Noboru Ito}
\address{The University of Tokyo, \\
3-8-1, Komaba, Meguro-ku, Tokyo 153-8914, Japan}
\email{noboru@ms.u-tokyo.ac.jp}
\author{Yusuke Takimura}
\address{Gakushuin Boys' Junior High School \\
1-5-1, Mejiro, Toshima-ku, Tokyo 171-0031, Japan}
\email{Yusuke.Takimura@gakushuin.ac.jp}
\title[A tabulation of knot projections]{The tabulation of prime knot projections with their mirror images up to eight double points}
\maketitle
\section{\textbf Introduction}
Arnold (\cite[Figure~53]{AA}, \cite[Figure~15]{A_book}) obtained a table of reduced knot projections (equivalently, reduced generic immersed spherical curves) up to seven double points.    
In Arnold's table, the number of prime knot projections with seven double points is six.  However, this table is incomplete (see Figure~\ref{arnold}, which equals \cite[Figure~53]{AA}). 
\begin{figure}[h!]
\centering
\includegraphics[width=11cm]{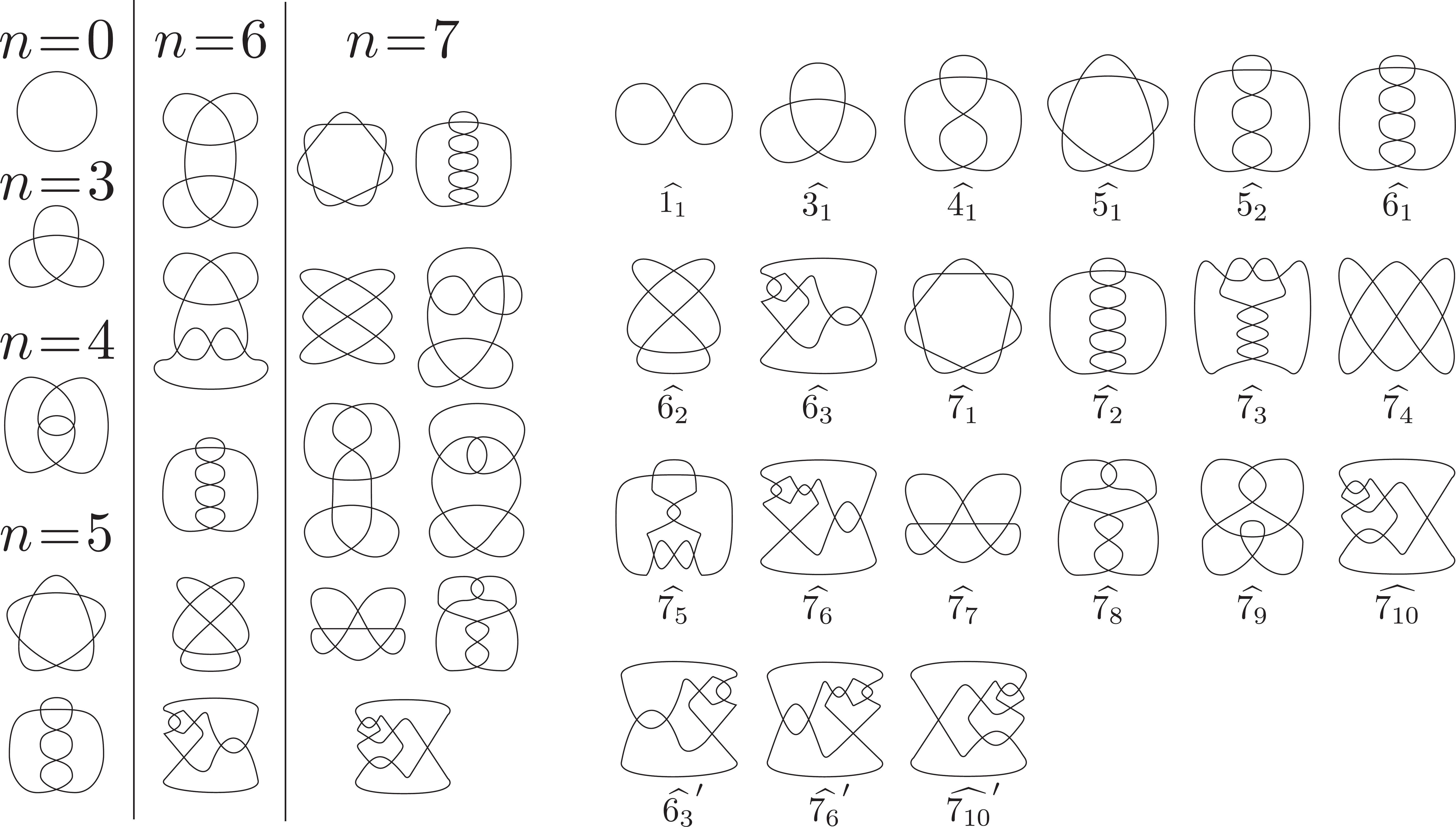}
\caption{Arnold's table  \cite[Figure~53]{AA} of knot projections with $n$ double points (left) and our table of prime knot projections up to seven double points (right).  The prime knot projections $\widehat{7_3}$, $\widehat{7_5}$, $\widehat{7_6}$, and $\widehat{7_9}$ are missing in Arnold's table.  }
\label{arnold}
\end{figure}
  
Nowadays, Arnold's table had been completed, e.g., by \cite{CHLR} that is a table, obtained by Gauss diagrams, up to ten double points.  However, the authors have not been able to find any table of knot projections with their mirror images (see Figure~\ref{arnold}).   
In this paper, we systematically construct the complete table of prime knot projections with their mirror images up to eight double points by flypes.  We tabulate knot projections using flypes in a way obeyed by an approach of Conway \cite{C} who studies rational tangles.  This paper provides the complete table of prime knot projections with their mirror images, without redundancy,  up to eight double points systematically thorough a finite procedure by flypes.  
In this paper, we show how to tabulate the knot projections up to eight double points by listing tangles with at most four double points by an approach of Conway.  
Also to tabulate knot projections up to ambient isotopy, we introduce arrow diagrams (oriented Gauss diagrams) of knot projections having no over/under information of each crossing, which were originally
introduced as arrow diagrams (oriented Gauss diagrams) of knot diagrams by Polyak and Viro \cite{PV}.  An arrow diagram (oriented Gauss diagram) completely detects the difference between a knot projection and its mirror image (Proposition~\ref{prop1}).  

In this paper, by a \emph{double point} we shall mean a transverse double point of a knot projection and by a \emph{crossing} we shall mean a double point with over/under information of a projection of a knot (Definition~\ref{def_reduced}). 

In \cite{A_book}, Arnold introduced the notion of a \emph{reducible} knot projection and he wrote: 

\vspace{3mm}
\emph{``Many of these irreducible curves are ``combinatorics" of simpler curves.  For instance, the first two curves with six crossings are two different combinations of two trefoil curves.  However, I do not know any formal theory describing such combinations."} 
\vspace{3mm}

In fact, Arnold's theory did not suggest notions of a \emph{prime} knot projection and a \emph{connected sum}.  
The notion describing ``such combination" by Arnold corresponds to the notion of connected sums, which is defined in this paper in a standard manner (Definition~\ref{connected}).  Every knot projection is one of prime knot projections or a connected sum of some prime knot projections.  The primeness is also defined in this paper (Definition~\ref{prime}).  
Arnold obtained a table of knot projections \cite{AA, A_book}.  He did not describe how to tabulate the knot projections.  
Dowker and Thistlewaite \cite{DT} explained an algorithm that in principal could generate all possible knot projections up to any crossing number.  Carrying out this algorithm depends on available computer power, and using this approach, the current knot table of prime knots up to 16 crossings has been assembled \cite{HTW}.    
One of the prior existing classical  knot tables of prime knots up to 10 crossings can be found in a book by Rolfsen \cite{R}.   

Here, we mention the difference between tabulations for knots and knot projections.  On one hand, if you would like to make a knot table with minimal $n$ crossings, you may apply the Dowker-Thistlewaite algorithm, arrange the over/under information in the all $2^n$ possible ways, and detect different pairs by using knot invariants.  On the other hand, if you would like to make a table of knot projections with minimal $n$ double points, first, enumerate alternating knots because it is known that an (reduced) alternating knot diagram of a given knot has the minimum number of crossings (Tait's conjecture \cite{murasugi}).  Second, it is well known that every knot projection uniquely determines (up to mirror symmetry) an alternating knot.  However, the set of knot projections with $n$ crossings is larger than the set of alternating knots with minimal $n$ crossings: all alternating knot diagrams obtained from a given one by a series of flypes correspond to the same knot.   
  
In this paper, through the use of flyping, we propose a systematic tabulation of prime knot projections and give the table of prime knot projections up to eight double points by flypes.   
Here, note that every nontrivial knot projection consists of two tangles.  In this paper, we show how to tabulate knot projections up to eight double points by using tangles with at most four double points.  We expect to extend our approach to a general case later.

As described above, our tabulation approach is different from other approaches.  
Our tabulation approach is basically obeyed by the approach of Conway\cite{C} (for rational tangles) and is a method for drawing the knot projection (thus, there is no need with verification of the realizability of a given code).  
For tabulating tangles with over/under information, cascade diagrams were used in \cite{B} and graphs were used in \cite{K}.   These methods are also different from the one proposed in our paper because the two methods do not use flype theory, which is what our approach is based on.  In particular, for a given knot projection of an alternating knot, we show how to obtain the other projections of the knot by using tangles with the smaller number of crossings.  Recently, Harrison has assembled a table of four regular graphs up to 10 double points \cite{CH}.  
It is also necessary to say that Knotscape software will identify the prime knot type of any given prime knot projection with at most $16$ crossings and thus in some sense Knotscape contains all knot projections of these knots.  While Knotscape does not handle composites directly, the methods in Knotscape can deal with composites and their diagrams up to $16$ crossings just fine.     

The novel approach in this paper is to use tangles and flypes in a systematical manner to tabulate knot projections.  Finally, we introduce an arrow diagram obtained from a knot projection that allow us to construct a complete list of mirror images for a given set of knot projections.  

\section{\textbf Preliminaries}
\begin{definition}[knot, knot projection, knot diagram]\label{def_reduced}
A \emph{knot} is the image of a smooth embedding from $S^1$ to $\mathbb{R}^3$.      
A \emph{knot projection} is the image of generic immersion into an oriented $2$-sphere.  Each self-intersection is a transverse double point.  
Let $P$ be a knot projection.  The \emph{mirror image} $P'$ of $P$ is $P$ with the orientation of the $2$-sphere reversed.  Then, we say that we consider $P$ \emph{up to mirror symmetry} if we identify $P$ with $P'$ depending on situations.  
Let $P$ and $P'$ be knot projections where $P$ is ambient isotopic to $P'$.  Then we say that $P=P'$.    
A \emph{knot diagram} is a knot projection where the two paths at each double point are assigned to be the over path and the under path respectively.  A double point of a knot diagram is called a \emph{crossing}.    
\end{definition}
\begin{definition}[tangle]
Let $T$ be the image of a generic immersion of two (one, resp.) interval(s) into $\mathbb{R} \times [0, 1]$ where the boundary points of the intervals map bijectively to the four (two, resp.) points 
\[
\{1, 2 \} \times \{ 0 \}, \{1, 2 \} \times \{ 1 \}  \qquad (\{1 \} \times \{ 0 \}, \{1 \} \times \{ 1 \},~{\textrm{resp}}.) 
\]
These four (two, resp.) points are called the \emph{endpoints} of $T$.    
If there exists an orientation-preserving embedding $\iota :$ $\mathbb{R}^2 \to S^2$, $\iota|_{T}(T)$ is called a \emph{tangle} ((1, 1)-tangle, resp.).  Then, the images of $\iota|_{\{ {\operatorname{endpoints}} \}}$ of the \emph{endpoints} of $T$ are called \emph{endpoints} of the  tangle.  
\end{definition}
By definition, there exists a sufficiently long interval $I$ ($\subset \mathbb{R}$) such that $T \subset I \times [0, 1]$ and we choose $\iota$ satisfying that $\iota|_{I \times [0, 1]} (I \times [0, 1])$ is orientation-preserving homeomorphic to a closed disk $D$.  In the rest of this paper, without loss of generality, we suppose that every tangle satisfies this condition.  The disk $D$ is called an \emph{ambient disk} of a tangle.  
\begin{definition}[connected sum of knot projections, prime tangle]\label{connected}
Let $P_1$ and $P_2$ be two knot projections.   
We choose an orientation of the ambient $2$-sphere of $P_i$ for each $i=1, 2$.  
Let $D_i$ be a $2$-disk ($\subset$ the oriented $S^2$) where the pair $(D_i, D_i \cap P_i)$ is pairwise-homeomorphic to the standard disk $({\mathbf D}^2, {\mathbf D}^1)$ for each $i=1, 2$.  Let $E_i$ be the $2$-disk satisfying $S^2$ $=$ $D_i \cup E_i$ and $D_i \cap E_i$ $=$ $\partial D_i$ $=$ $\partial E_i$ for each $i=1, 2$.  Let $S$ be a 2-sphere obtained form a disjoint copy of $E_1$ and $E_2$ by identifying $(\partial E_1, \partial E_1 \cap P_1)$ and $(\partial E_2, \partial E_2 \cap P_2)$ under an orientation reversing homeomorphism $h :$ $\partial E_1$ $\to$ $\partial E_2$ such that $h(\partial E_1)$ $=$ $\partial E_2$.  Then $(E_1 \cap P_1) \cup (E_2 \cap P_2)$ is a knot projection and is called a \emph{connected} sum of $P_1$ and $P_2$.  
A knot diagram obtained from a connected sum of two knot projections is called a connected sum of knots. 

By definition, a connected sum of two nontrivial knot projections is naturally decomposed into two (1, 1)-tangles, each of which has at least one double point.   
Let $T$ be a tangle with an ambient disk $D$.  Suppose that for any $D \subset D'$ that intersects the arcs of $T$ in a single curve $\alpha$, $\alpha$ is a simple arc.  Then $T$ is called a prime tangle.      
\end{definition}   
Classically a $2$-string tangle means either locally knotted or rational or prime.  Note that,  in our definition, we consider standard rational tangle projections as prime.  
\begin{definition}[trivial knot projection, prime knot projection]\label{prime}
Let $P$ be a knot projection.  
The knot projection with no double points is called the \emph{trivial knot projection}.  
Suppose that $P$ is not a connected sum of  nontrivial knot projections.  Then $P$ is called a \emph{prime} knot projection.   
\end{definition}
\begin{definition}[prime knot, alternating knot]
If a knot is not a connected sum of  nontrivial knots, it is called a \emph{prime knot}.  An \emph{alternating knot} is a knot with a knot diagram that has crossings that alternate between over and under as one travels around the knot in a fixed direction.   
\end{definition}  
\begin{definition}[flype of knot projections]\label{def_flype}
A \emph{flype} in a knot projection is an operation as shown in Figure~\ref{flype1}.  A flype that does not change a knot projection (up to ambient isotopy of the projection) is called a trivial flype.  A flype is called a \emph{nontrivial flype} if a flype is not a trivial flype.
The application of finitely many flypes is called \emph{flyping}.    
\begin{figure}[h!]
\includegraphics[width=4cm]{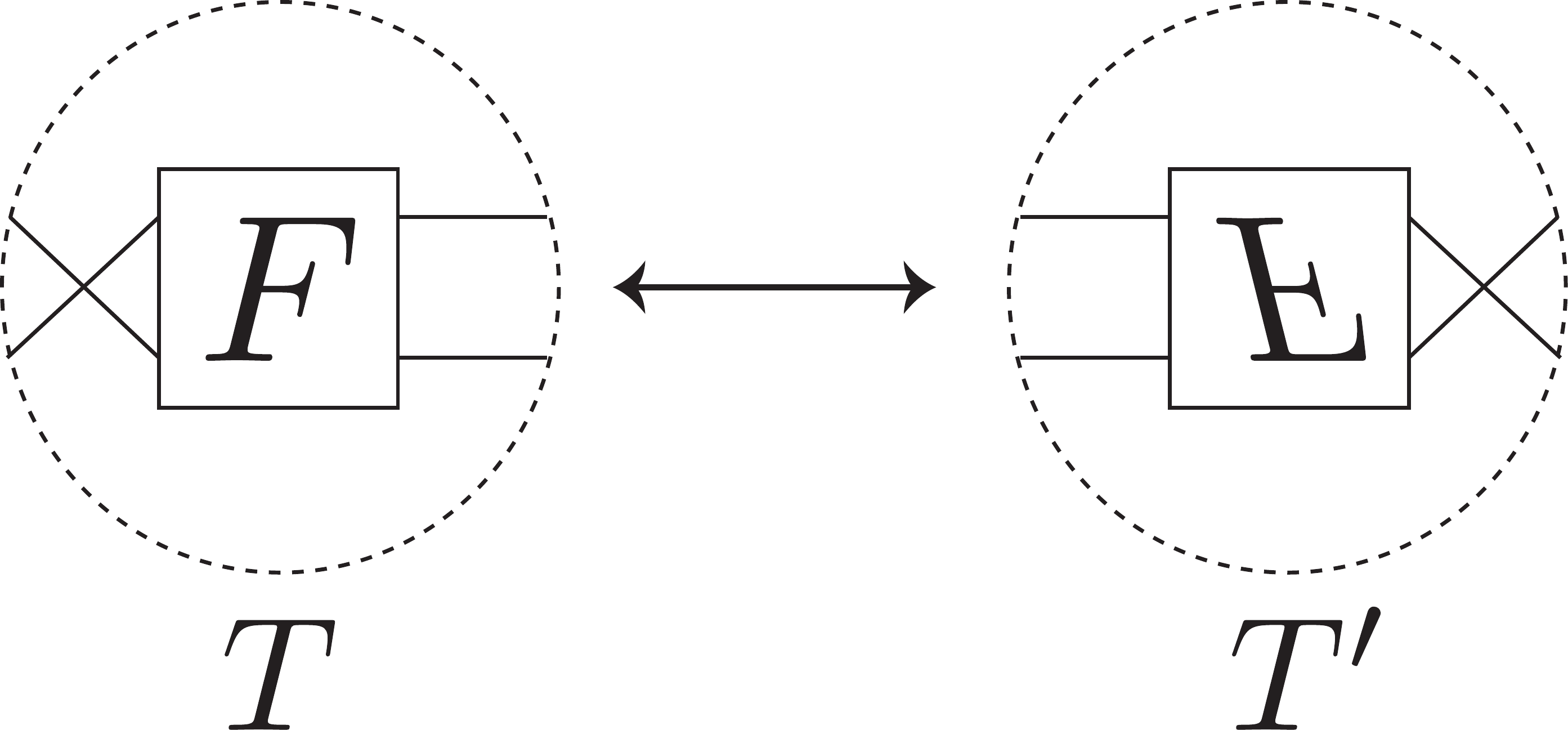}
\caption{Flype of knot projections.}\label{flype1}
\end{figure}
Suppose that we apply flyping to a knot projection $P$, and the resulting knot projection $P$ satisfies $P \neq P'$.  Then the flyping is called
 \emph{nontrivial flyping}.  
\end{definition}
A flype for knot diagrams is defined by  Figure~\ref{flype} in the same way as Definition~\ref{def_flype}. 
\begin{figure}[h!]
\includegraphics[width=8cm]{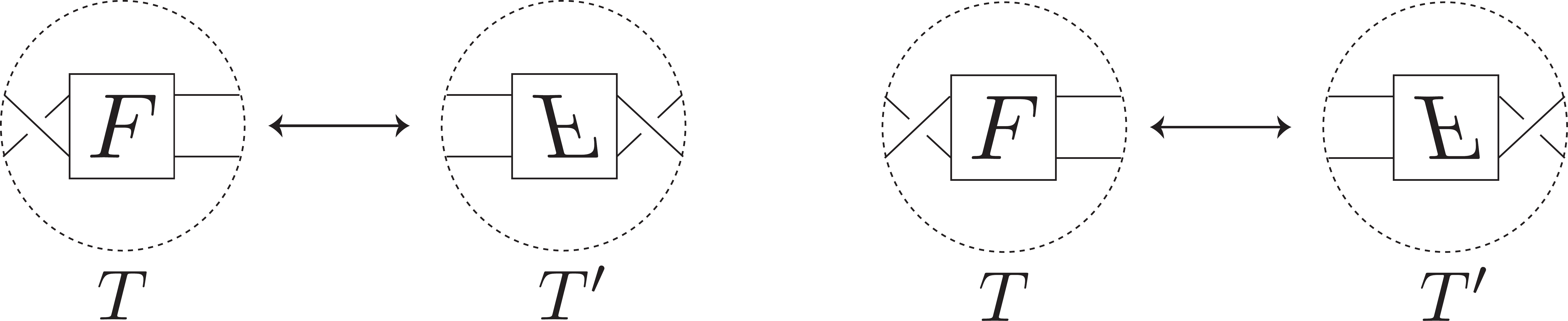}
\caption{Flype of knot diagrams.}\label{flype}
\end{figure}
\begin{notation}\label{notation1}
We use traditional notations $N(\cdot)$ or $N(T_1 + T_2)$ as in \cite{D} where $N(T)$ is the numerator of a tangle $T$ and for $T_1$ $+$ $T_2$, $+$ means a tangle addition.  
By a slight abuse of a notation, for flypes, we use the same notation for knot projections as that of knot diagrams, $P$ $=$ $N(A+1+B)$ and the tangle is denoted by $A+1+B$.  
Then, for every flype in a knot projection $P$, it is easy to see that $P$ is decomposed into three tangles, as shown in Figure~\ref{N(AC)}, which are denoted by $A$, $(+1)$, and $B$ from the left. 
By a slight abuse of a notation of knot diagrams, if a flype is rotating $A$ ($B$, resp.),    
this flype is called a \emph{flype of a crossing across the tangle $A$ $($$B$, resp.$)$}.  Then, we mean that we replace $1+A$ ($B+1$, resp.) with $A+1$ ($1+B$, resp.).    
\end{notation}
\begin{figure}[h!]
\includegraphics[width=4cm]{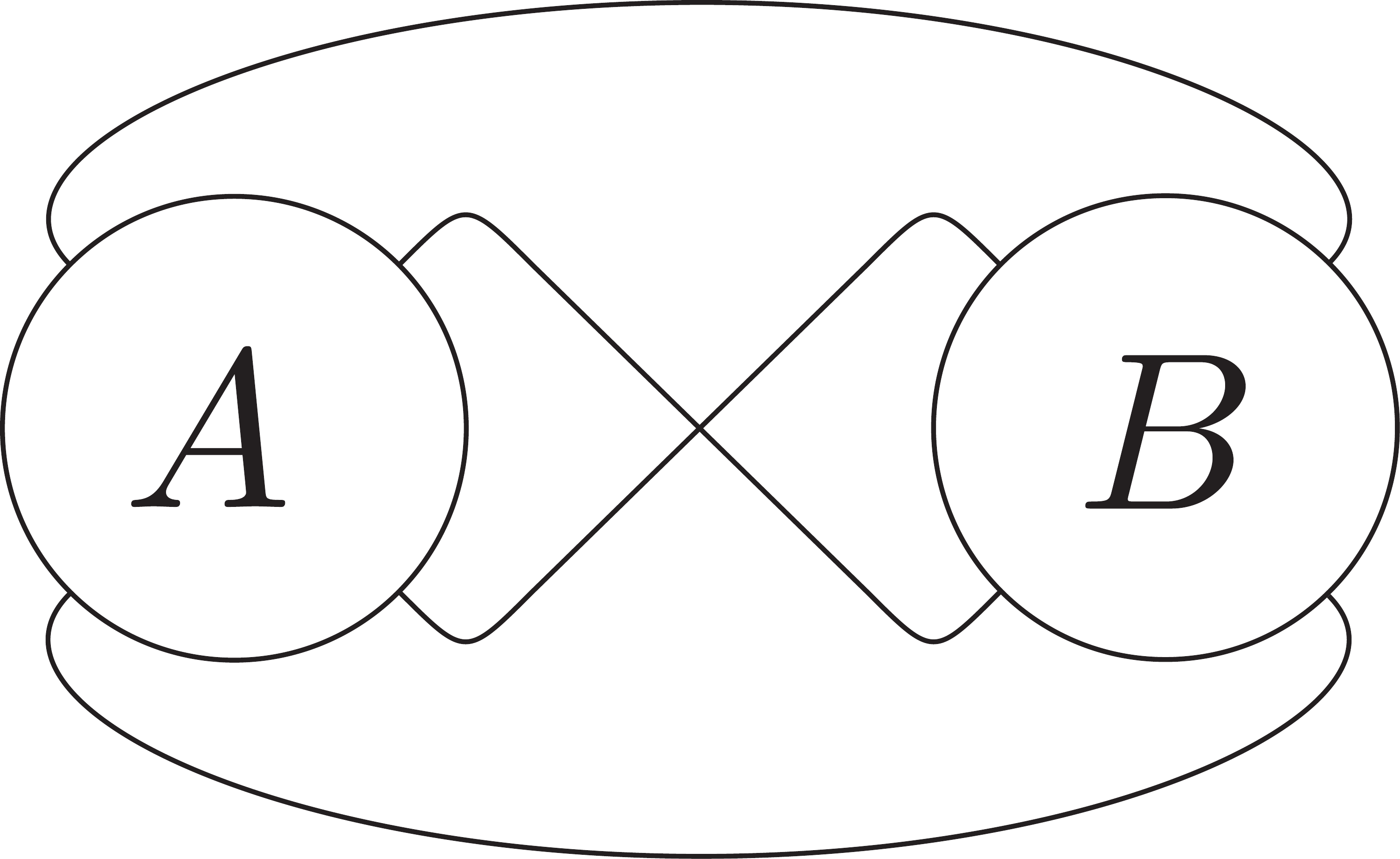}
\caption{$N(A+1+B)$.}\label{N(AC)}
\end{figure}
\begin{definition}
We use traditional terminologies as in \cite{ES} (we also see \cite{D}).   
For the standard circle parametrized by $r e^{i \theta}$, the four endpoints will be fixed at $NW$ $=$ $r e^{5 i \pi /4}$, $NE$ $=$ $r e^{i \pi /4}$, $SW$ $=$ $r e^{-5 i \pi /4}$, and $SE$ $=$ $r e^{- i \pi /4}$.  If NW and SW (NE, SE, resp.) are connected, then the tangle is said to be a parity $(\infty)$ ((0), (1), resp.) tangle (e.g., see Figure~\ref{tangle1}).  
\end{definition}
\begin{notation}\label{notation2}
In this paper, we consider tangles up to a rotation of a multiple of $\pi /4$.  Thus, parity $(\infty)$ and parity (0) tangles are the same and one can drop the usual (NW, SW, NE, and SE) boundary designations.  
Then, a parity $(\infty)$, (0), or (1)  \emph{prime} tangle is denoted by $\widehat{T_n}$  or $\widehat{U_n}$, as shown in Figure~\ref{tangle1}, where $n$ is an index which represents the number of double points of the tangle.      
\begin{figure}
\includegraphics[width=8cm]{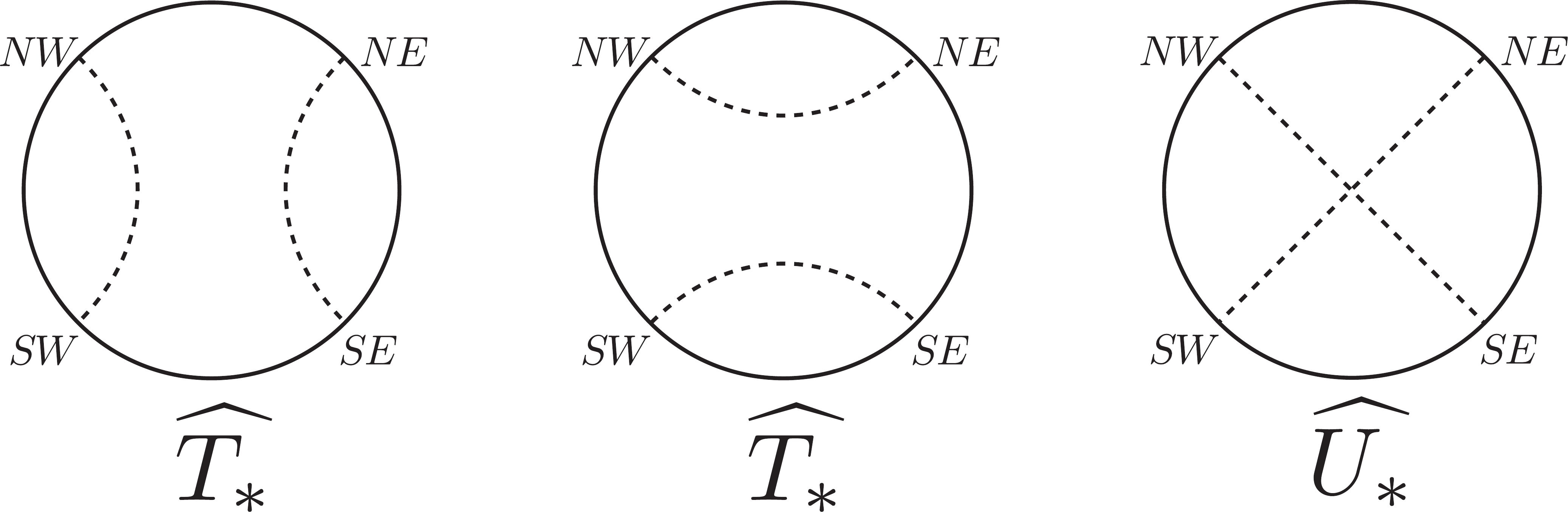}
\caption{Dotted curves indicate the connections of tangles.  From the left, each tangle is called a parity $(\infty)$ tangle $T_*$, a parity $(0)$ tangle $T_*$, and a parity $(1)$ tangle $U_*$.  }\label{tangle1}
\end{figure}
\end{notation}
\section{\textbf The main result and a conjecture}
\begin{notation}[knot projection $\widehat{n_i}$]\label{notation3}
Let $n$ be the number of double points of a knot projection and let $i$ be a positive integer.   
The symbol $\widehat{n_i}$ ($n \le 8$) denotes a knot projection defined as follows:
\begin{itemize}
\item For $n \le 6$, $\widehat{n_i}$ corresponds to the knot diagram $n_i$ in the knot table in \cite{R}.   
\item For $n = 7$ ($i \le 7$), $\widehat{n_i}$ corresponds to the knot diagram $n_i$ in the knot table in \cite{R}.
\item For $n = 7$ ($8 \le i \le 10$),   $\widehat{7_{i}}$ is the knot projection obtained from $\widehat{7_{i-3}}$ by a flype.
\item For $n = 8$ ($i \le 18$), $\widehat{n_i}$ corresponds to the knot diagram $n_i$ in the knot table in \cite{R} (note that each $8_i$ ($1 \le i \le 18$) represents an alternating knot diagram). 
\item For $n = 8$ ($19 \le i \le 27$), $\widehat{8_{i}}$ is the knot projection obtained from $\widehat{8_{j}}$ ($1 \le j \le 18$) by at most two flypes.  In the following table, $N_f$ denotes the minimal number of flypes necessary to deform from $\widehat{8_{j}}$ to $\widehat{8_{i}}$.  
\begin{center}

{\renewcommand\arraystretch{1.4}

\begin{tabular}{|c|c|c|c|c|c|c|c|} \hline
\text{knot projection} & $\widehat{8_{6}}$ & $\widehat{8_{8}}$ & $\widehat{8_{11}}$ & $\widehat{8_{12}}$ & $\widehat{8_{13}}$ & $\widehat{8_{14}}$ & $\widehat{8_{15}}$ \\ \hline
$N_f = 1$ & $\widehat{8_{19}}$ & $\widehat{8_{20}}$ & $\widehat{8_{21}}$ & $\widehat{8_{22}}$ & $\widehat{8_{24}}$ & $\widehat{8_{25}}$, $\widehat{8_{26}}$ & $\widehat{8_{27}}$ \\ \hline
$N_f = 2$ &  &  &  & $\widehat{8_{23}}$ &  &  &  \\ \hline
\end{tabular}
}
\end{center}
\item For every $n$, let $\widehat{n_i}'$ be the mirror image of $\widehat{n_i}$.  
\end{itemize}
\end{notation}
\begin{definition}\label{def2}
Let $\mathcal{P}_{\le n}$ be the set of prime knot projections up to orientations of the ambient $2$-sphere with at most $n$ double points. 
Let $\mathcal{P}'_{\le n}$ be the set of knot projections, each of which is the mirror image of each element of $\mathcal{P}_{\le n}$.       
\end{definition}
\begin{theorem}\label{mainthm}
Let $\mathcal{P}_{\le 8}$ and $\mathcal{P}'_{\le 8}$ be the set as in Definition~\ref{def2}.  
The set $\mathcal{P}_{\le 8} \cup \mathcal{P}'_{\le 8}$ is given in Table~\ref{table1} and Table~\ref{table4}.  
\begin{table}
\caption{The table of prime knot projections in $\mathcal{P}_{\le 8}$ up to eight double points ($\widehat{n_i}$ is the symbol as in  Notation~\ref{notation3}).}\label{table1}
\includegraphics[width=10cm]{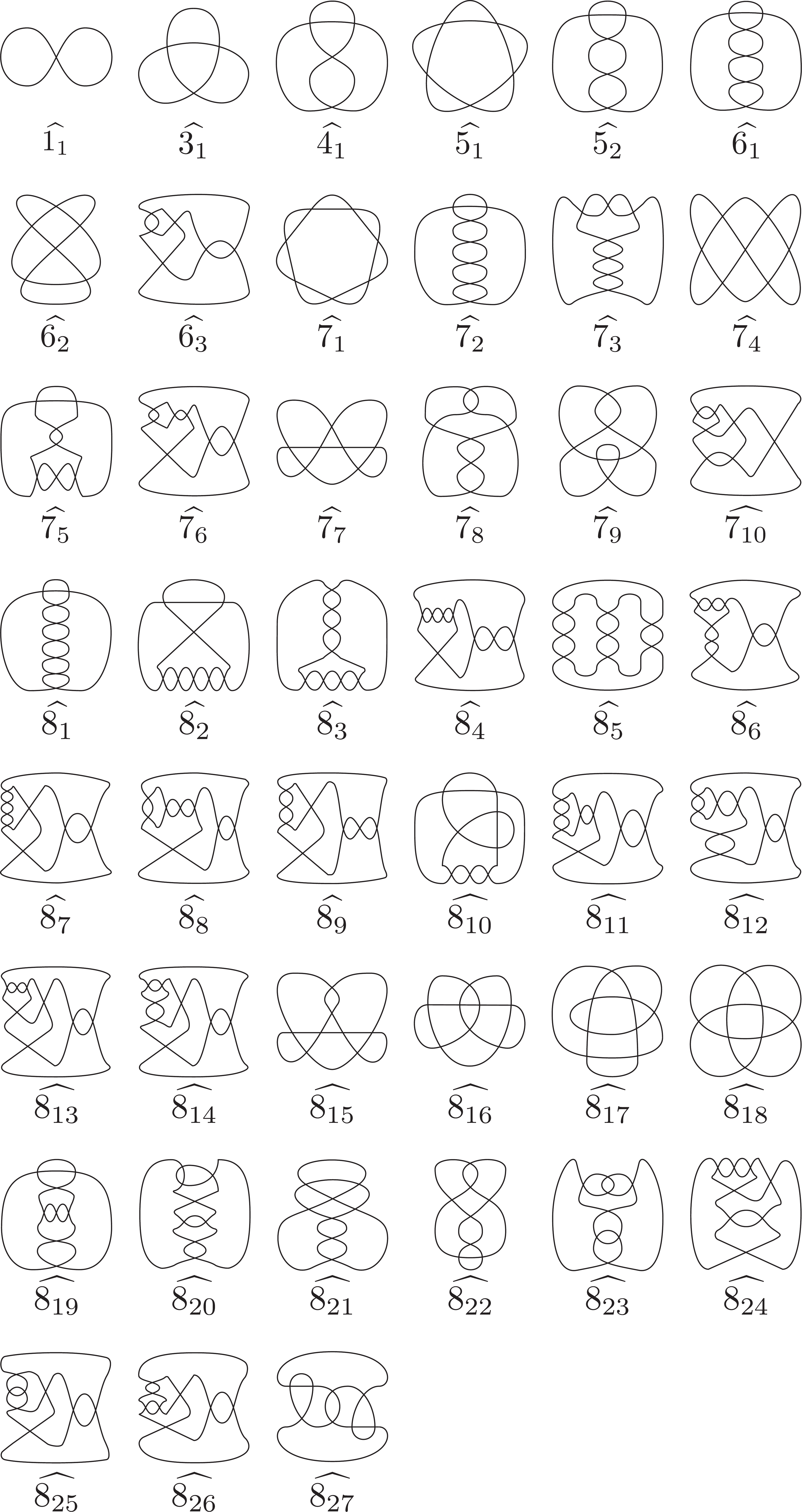}
\end{table}  
\end{theorem}
For a proof of Theorem~\ref{mainthm}, see Section~\ref{proofm}.  
\begin{table}
\caption{The table of knot projections in $\mathcal{P}'_{\le 8} \setminus \mathcal{P}_{\le 8} $ ($\widehat{n_i}$ is the symbol as in  Notation~\ref{notation3}).}\label{table4}
\includegraphics[width=12cm]{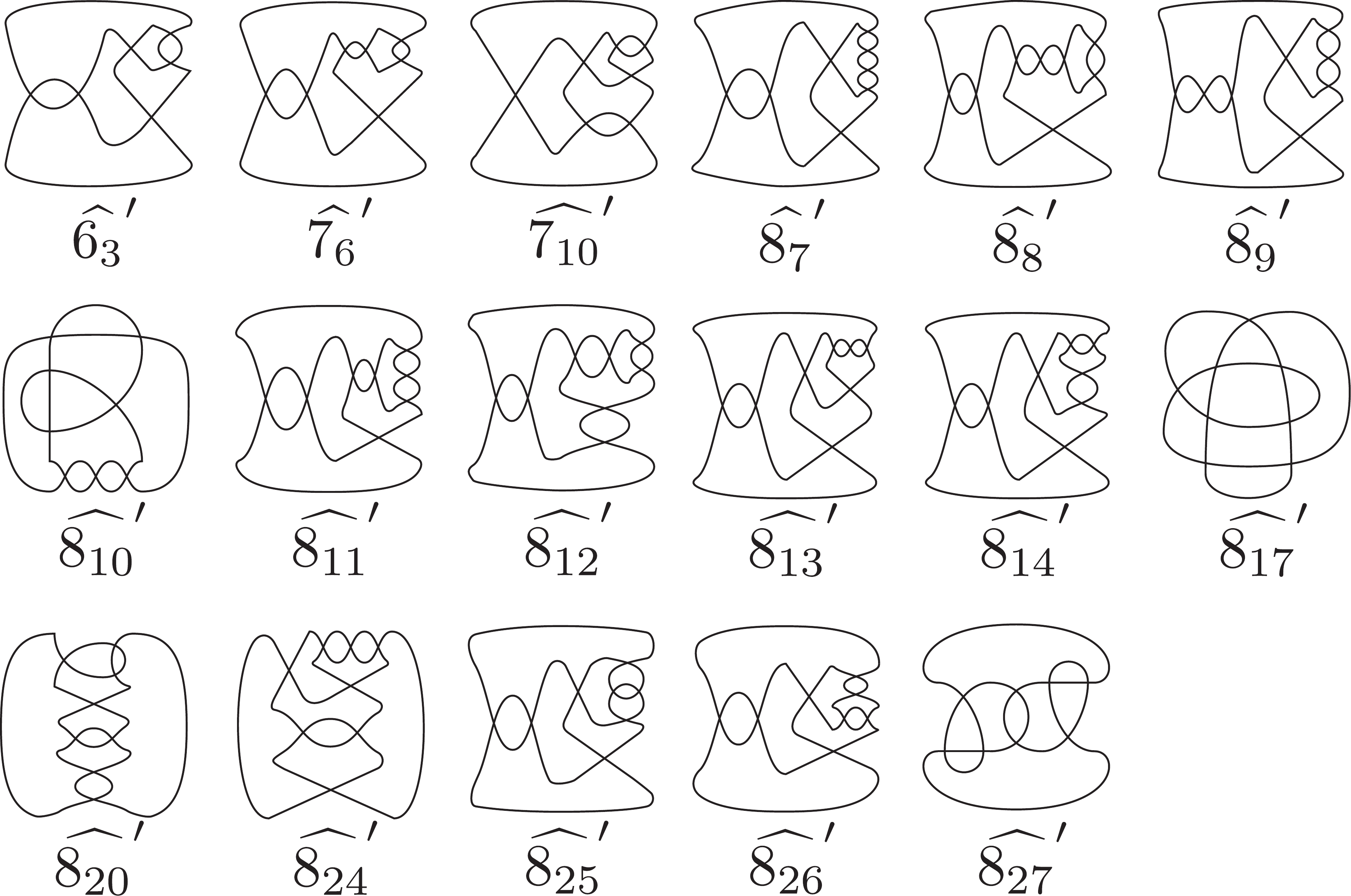}
\end{table}
By Theorem~\ref{mainthm}, prime knots are compared with prime knot projections up to eight double points as follows.   

{
\begin{center}
\begin{tabular}{|c|c|c|c|c|c|c|c|c|} \hline
$n$ & 1 & 2 & 3 & 4 & 5 & 6 & 7 & 8  \\ \hline
$|\mathcal{K}_n|$ & 0 & 0 & 1 & 1 & 2 & 3 & 7 & 21 \\ \hline
$|\mathcal{P}_n|$ & 1 & 0 & 1 & 1 & 2 & 3 & 10 & 27 \\ \hline
\end{tabular}
\end{center}
}

\begin{conjecture}\label{mainconj}
Let $n$ be a positive integer.  
Let $P$ be a prime knot projection and $K$ a prime knot.  
Let $c(P)$ be the number of double points of $P$ and $c(K)$ the minimum number of crossings of $K$.  Let $\mathcal{K}_n$ $=$ $\{ K~|~c(K)=n \}$ and $\mathcal{P}_n$ $=$ $\{ P~|~c(P)=n \}$.  
For a set $S$, $|S|$ denotes the cardinality of $S$.   
\begin{enumerate}
\item If $3 < n < m$, $|\mathcal{K}_n| < |\mathcal{K}_m|$ $($a famous conjecture \cite[Page~34, Unsolved Question~4]{Adams}$)$.   
\item If $3 < n < m$, $|\mathcal{P}_n| < |\mathcal{P}_m|$.  
\item $|\mathcal{K}_n| \le |\mathcal{P}_n|$. \label{conj3}  
\end{enumerate}
\end{conjecture}

Conjecture~\ref{mainconj}~(\ref{conj3}) is not obvious, for example, for a knot projection $P$ in Figure~\ref{hana}, there are at least three distinct knots.      
\begin{figure}[h!]
\includegraphics[width=8cm]{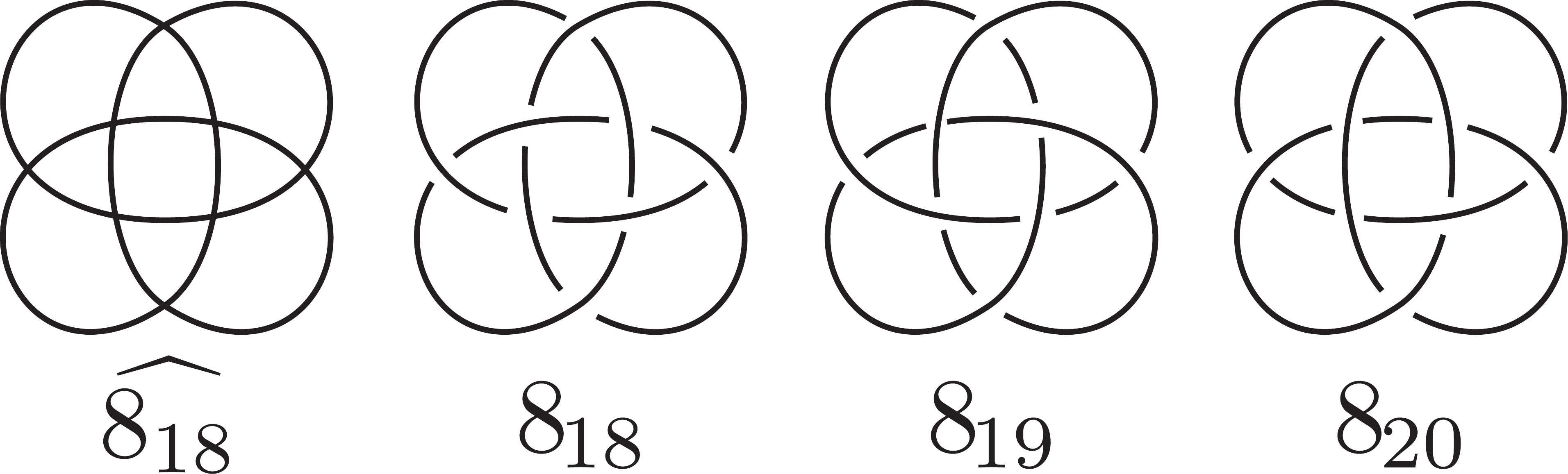}
\caption{Each of three distinct knots $8_{18}$, $8_{19}$, and $8_{20}$ in the knot table (up to isotopy) in \cite{R} has the knot projection $\widehat{8_{18}}$.}\label{hana}
\end{figure}  
\section{\textbf Proof of Theorem~\ref{mainthm}.  }\label{proofm}
Recall the following well-known facts: 
    
(1)~For every knot projection $P$, there exists a  knot diagram $D$ such that $P$ is obtained from $D$ by forgetting over/under information.    

(2)~Two alternating knots $K_1$ and $K_2$ are isotopic if and only if any two corresponding minimal knot diagrams of $K_1$ and $K_2$ are related by a finite sequence of flypes (Tait flyping conjecture, Theorem of Menasco and Thistlethwaite \cite{MT}).     

\subsection{\textbf Step~1: Tabulation of tangles at most four double points}
Recall that Notation~\ref{notation1}.  By the definition of flypes (Definition~\ref{def_flype}), we have Lemma~\ref{lemma0}.   
\begin{lemma}\label{lemma0}
For every flype in a knot projection $P$, $P$ is decomposed into two tangles $A$, $B$, and the third tangle $(+1)$ that satisfy $P$ $=$ $N(A+1+B)$.  Then, there are two choices, we can flype by either rotating $A$ or by rotating $B$.  Then, either choice results into the same knot projection up to mirror symmetry.     
\end{lemma}
In the rest of this paper, we suppose that for every knot projection $P$, the number of double points of $P$ is at most eight.   The statement of Lemma~\ref{lemma2} is given using Notation~\ref{notation2}.  
\begin{lemma}\label{lemma2}
For a prime knot projection $P$, suppose that $P'$ is obtained from $P$ by a flype of a crossing across a tangle $A$ or $B$ of $P$.  If the flype is a nontrivial flype, then, either $A+1$ or $1+B$ is a parity $(\infty)$ tangle $\widehat{T_*}$ with at most four double points.   
\end{lemma}
\begin{proof}
Let $P$ be a knot projection with a decomposition such that $P$ $=$ $N(A+1+B)$.  Note that $A$ and $B$ are prime tangles since $P$ is a prime knot projection.  
The only parity~(1) prime tangles $U$ with at most three double points are shown in Figure~\ref{U}.  
For $\widehat{U_1}$, it is clear that the flype of a crossing $\widehat{U_1}$ across a tangle is a trivial flype.  
For $\widehat{U_3}$, the flype of a crossing across a tangle $\widehat{T_2}$ ($\subset$ $\widehat{U_3}$) is a trivial flype.  
\begin{figure}[htbp]
\includegraphics[width=3cm]{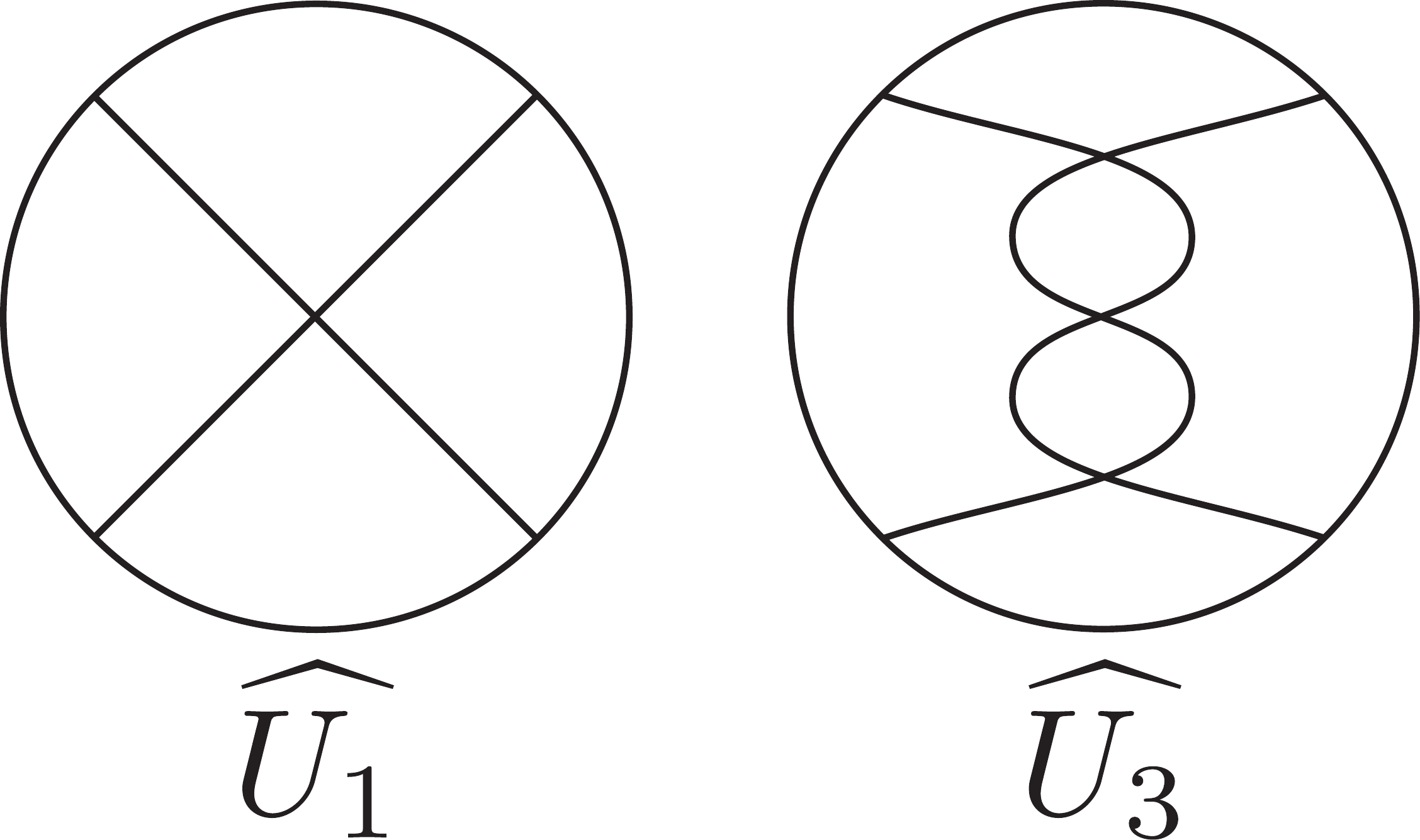}
\caption{$\widehat{U_1}$ and $\widehat{U_3}$.}\label{U}
\end{figure}  
Thus, it is sufficient to consider the case that both $A+1$ and $1+B$ are  parity ($\infty$) tangles, each of which is $\widehat{T}_{*}$ with at most four double points.
\end{proof}

It is easy to prove Lemma~\ref{lemma1} and we leave the details to the reader.   
\begin{lemma}\label{lemma1}
Every possibility of a prime parity $(\infty)$ tangle $\widehat{T_*}$ with at most four double points is one of the list of Figure~\ref{lemma3_fig}.        
\begin{figure}[htbp]
\includegraphics[width=11cm]{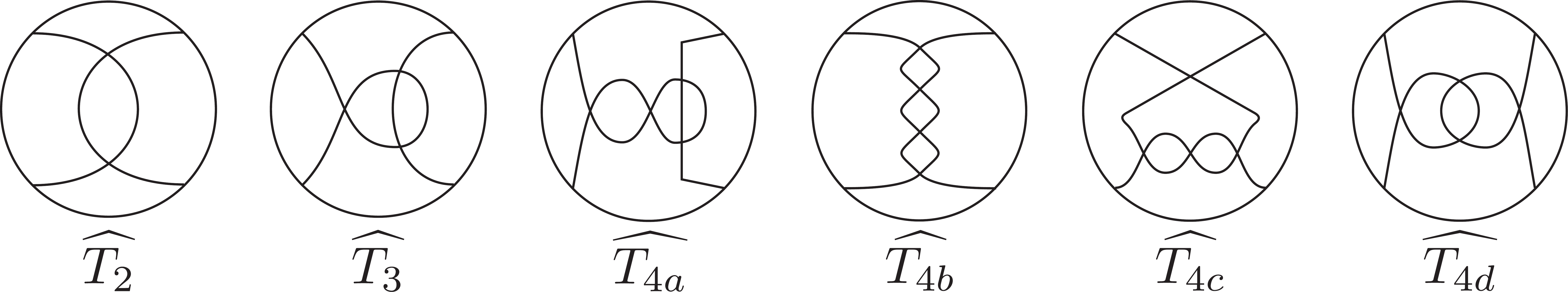}
\caption{$\widehat{T_2}$--$\widehat{T_{4d}}$.}\label{lemma3_fig}
\end{figure}
\end{lemma}

Lemma~\ref{lemma2} and Lemma~\ref{lemma1} imply Lemma~\ref{lemma5}.  
\begin{lemma}\label{lemma5}
Every nontrivial flyping is generated by flypes, each of which is a flype of a crossing across the tangle $\widehat{T_{2}}$ or $\widehat{U_{3}}$, as shown in  Figures~\ref{U} and \ref{lemma3_fig}.  
\end{lemma}
\begin{proof}
For a knot projection $P$, if there is a flype possible, then $P$ has a decomposition $P$ $=$ $N(A+1+B)$.  
This fact together with Lemma~\ref{lemma1} implies that it is sufficient to consider a tangle $A+1$ with at most four double points ($A$ has at most three double points) of type $\widehat{T_{*}}$ listed in Figure~\ref{lemma3_fig}.  This $A$ is $\widehat{T_{2}}$, $\widehat{T_{3}}$, or $\widehat{U_{3}}$.  

Here, note that we can exclude $\widehat{T_{3}}$.   This is because, in Figure~\ref{lemma3_fig}, a flype of a crossing across the tangle $\widehat{T_{3}}$ is generated by that of $\widehat{T_{2}}$, or is a trivial flype.  
\end{proof}

\subsection{\textbf Step~2: Tabulation of knot projections by flypes}
Recall the following notations and facts.  
Let $c(P)$ be the number of double points of $P$.
Let $\mathcal{P}_n$ $=$ $\{$ $P :$ prime  $|~c(P)=n$ $\}$.
Let $Alt_n$ be the set of knot projections, each of which is a projection of an alternating knot diagram, up to mirror symmetry, with $n$ crossings in the knot  table in \cite{R}.  
By using facts~(1) and (2) in the beginning of Section~\ref{proofm},  
$\mathcal{P}_n$ is obtained from $Alt_n$ via flypes. 
\subsection{\textbf Step~2a: Up to six double points}
A table of $\{ \widehat{1_1} \}$ $\cup$ $Alt_3$ $\cup$ $Alt_4$ $\cup$ $Alt_5$ $\cup$ $Alt_6$ is known as Figure~\ref{6}.  
\begin{figure}[h!]
\includegraphics[width=8cm]{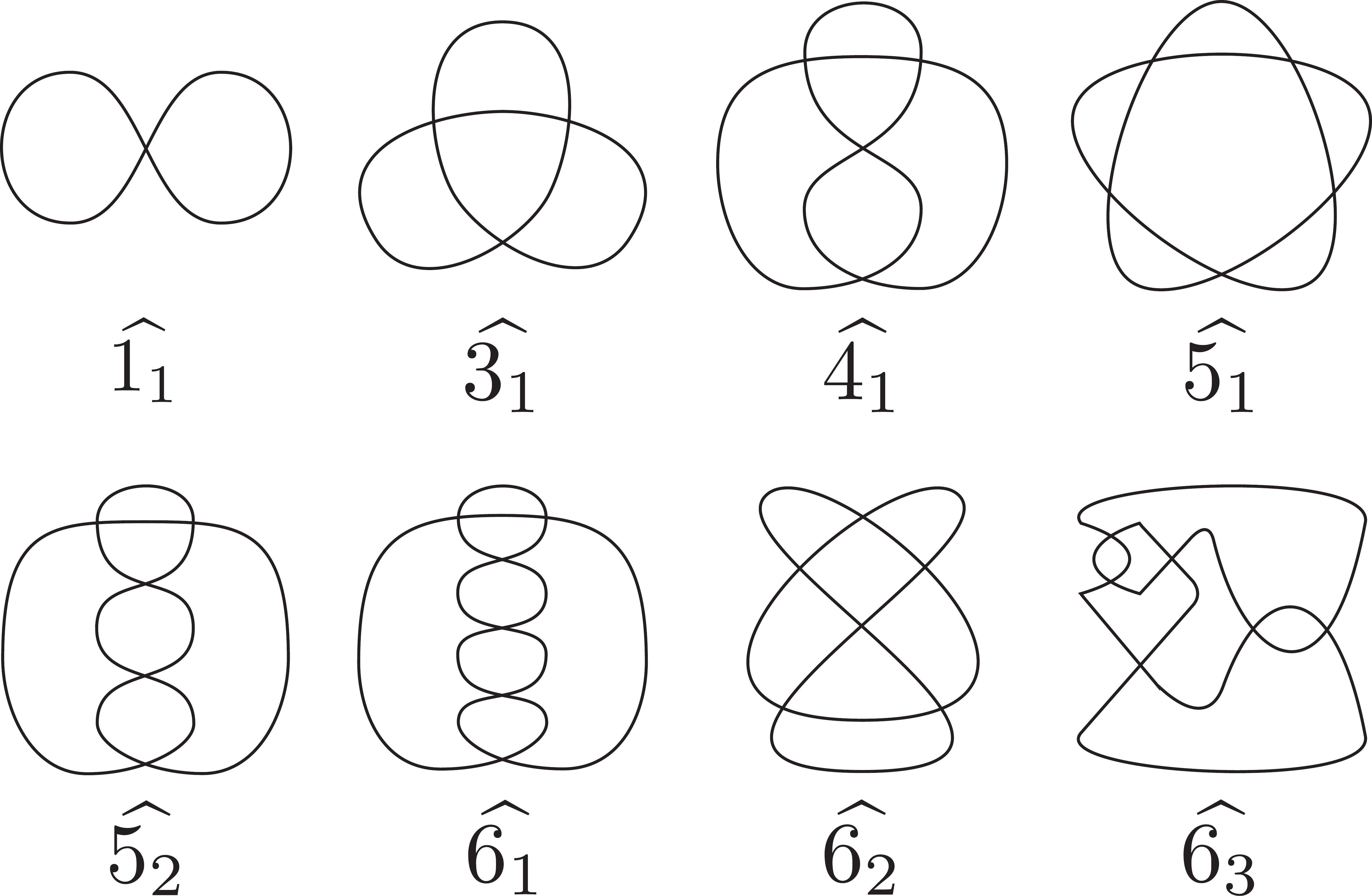}
\caption{$\{ \widehat{1_1} \}$ $\cup$ $Alt_3$ $\cup$ $Alt_4$ $\cup$ $Alt_5$ $\cup$ $Alt_6$.}\label{6}
\end{figure}
By Lemma~\ref{lemma0} and Lemma~\ref{lemma5}, for a knot projection, if there exists a nontrivial flyping which is caused, there exists $\widehat{T_{2}}$ $($$\subset$ $\widehat{T_3}$ $=1+\widehat{T_{2}}$, $\widehat{T_{2}}+1)$ for the knot projection.  
Thus, for each $P \in Alt_6$, if a knot projection $P'$ is obtained from $P$ by applying a flype of a crossing across the tangle $\widehat{T_{2}}$ and $P' \neq P$, $P' \in \mathcal{P}_6$.  However, there is no such tangle up to mirror symmetry ($\widehat{6_3}'$ is obtained from $\widehat{6_3}$ by a flype of a crossing across the tangle $\widehat{T_{2}}$).  Thus, $\mathcal{P}_6$ $=$ $Alt_6$.    
\subsection{\textbf Step~2b: Up to seven double points}
A table $Alt_7$ is known as Figure~\ref{7}.  
\begin{figure}[h!]
\includegraphics[width=8cm]{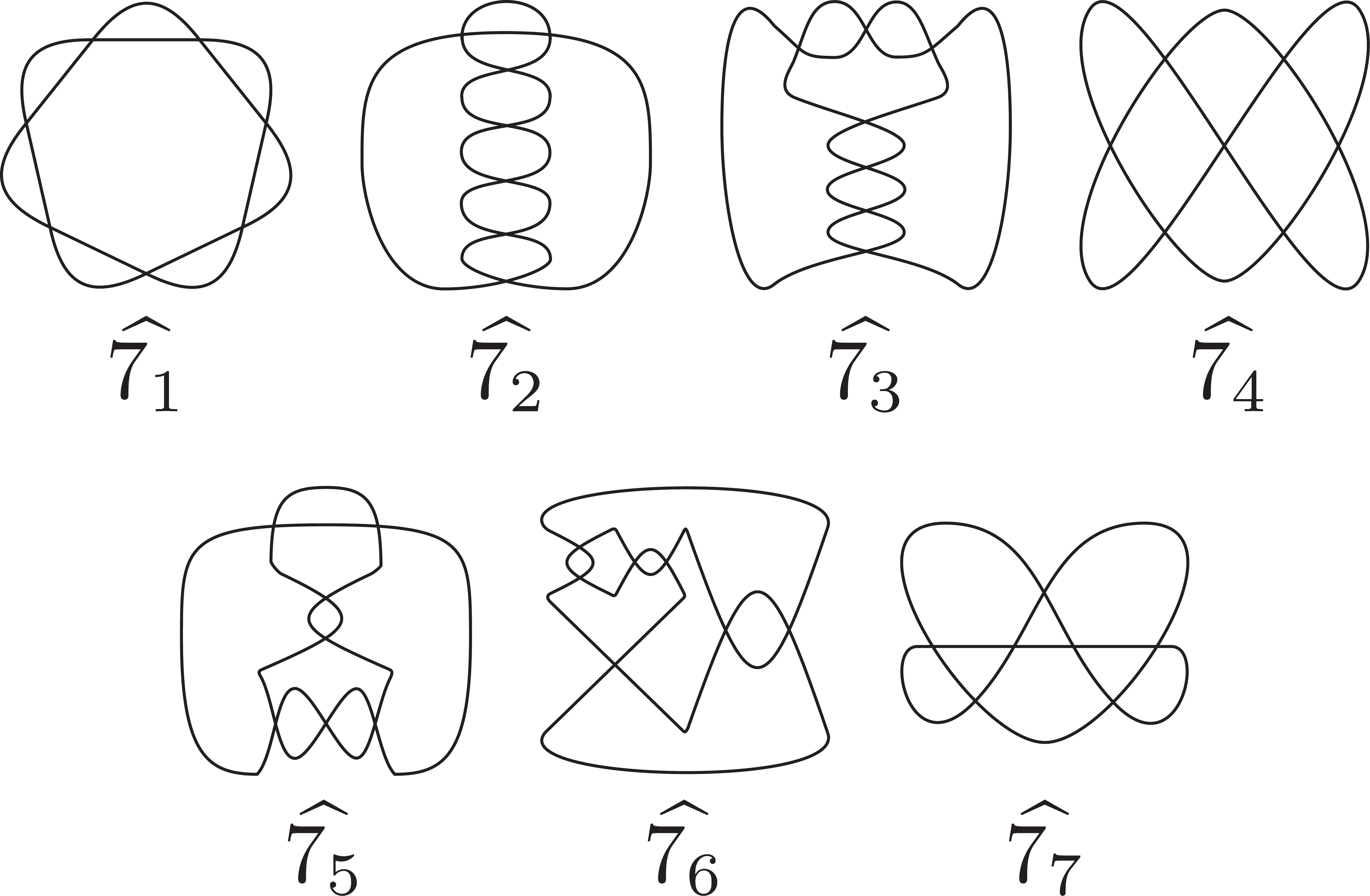}
\caption{$Alt_7$.}\label{7}
\end{figure}
By Lemma~\ref{lemma0} and Lemma~\ref{lemma5}, for a knot projection, if there exists a nontrivial flyping which is caused, there exists $\widehat{T_{2}}$ $($$\subset$ $\widehat{T_3}$ $=1+\widehat{T_{2}}$, $\widehat{T_{2}}+1)$ for the knot projection.   
Thus, for each $P \in Alt_7$, if a knot projection $P'$ is obtained from $P$ by applying a flype of a crossing across the tangle $\widehat{T_{2}}$ and $P' \neq P$,  $P' \in \mathcal{P}_7$.    
 
For example, we explain the first line of Figure~\ref{7flype} with respect to $\widehat{7_{8}}$ and $\widehat{7_5}$.  In Figure~\ref{7flype1}, (1) denotes an existence of a flype of a crossing across the tangle $\widehat{T_2}$ ($\subset  \widehat{T_3}$,  which equals $1+\widehat{T_2}$).  
\begin{figure}[h!]
\centering
\includegraphics[width=10cm]{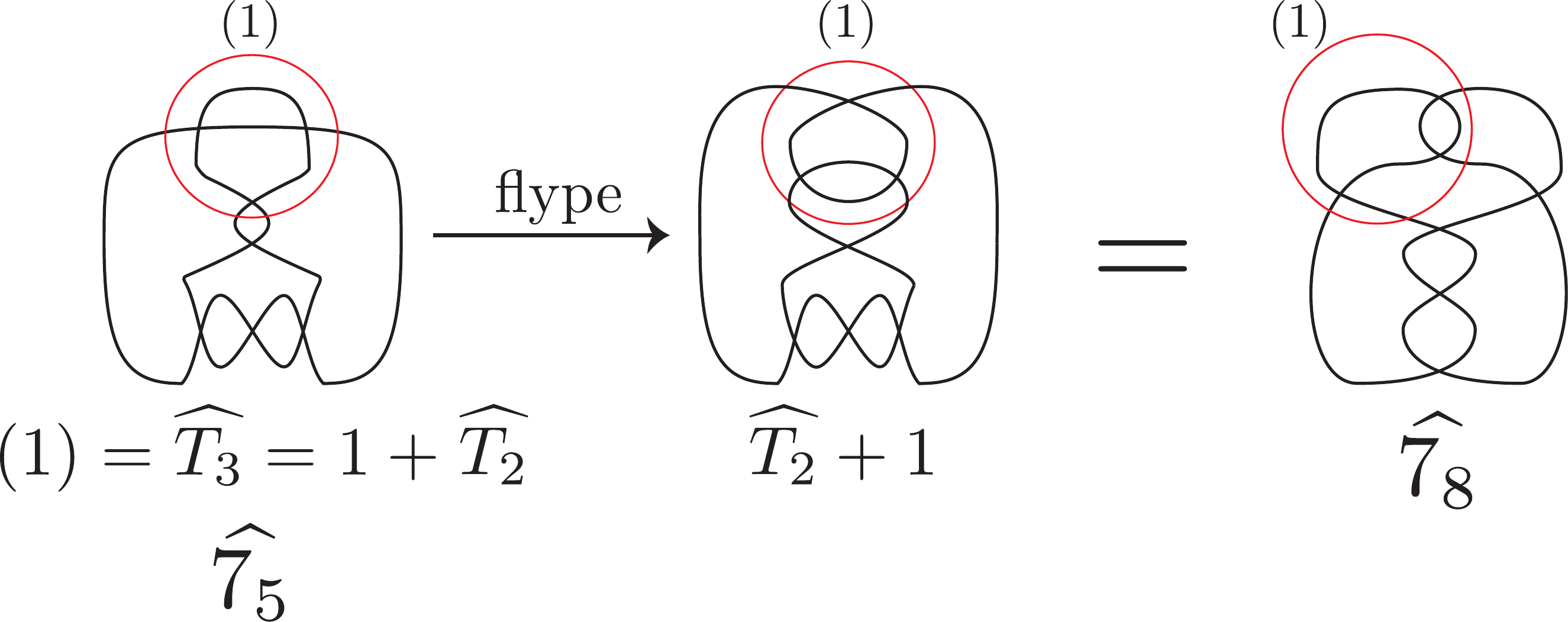}
\caption{$\widehat{7_{8}}$ is obtained from $\widehat{7_5}$. }\label{7flype1}
\end{figure}
Then, $\widehat{7_8}$ is obtained from $\widehat{7_5}$ up to ambient isotopy, as shown in Figure~\ref{7flype1}, which implies that $\widehat{7_8} \in {\mathcal{P}_7}$.     

Similarly, we list all the possibilities, i.e., for $\widehat{7_5}$, $\widehat{7_6}$, and $\widehat{7_7}$, there exist three ambient disks, each of which corresponds to a nontrivial flyping, as shown in Figure~\ref{7flype}, which implies that $\widehat{7_8}$, $\widehat{7_9}$, and $\widehat{7_{10}} \in {\mathcal{P}_7}$.  
  
\begin{figure}[h!]
\centering
\includegraphics[width=10cm]{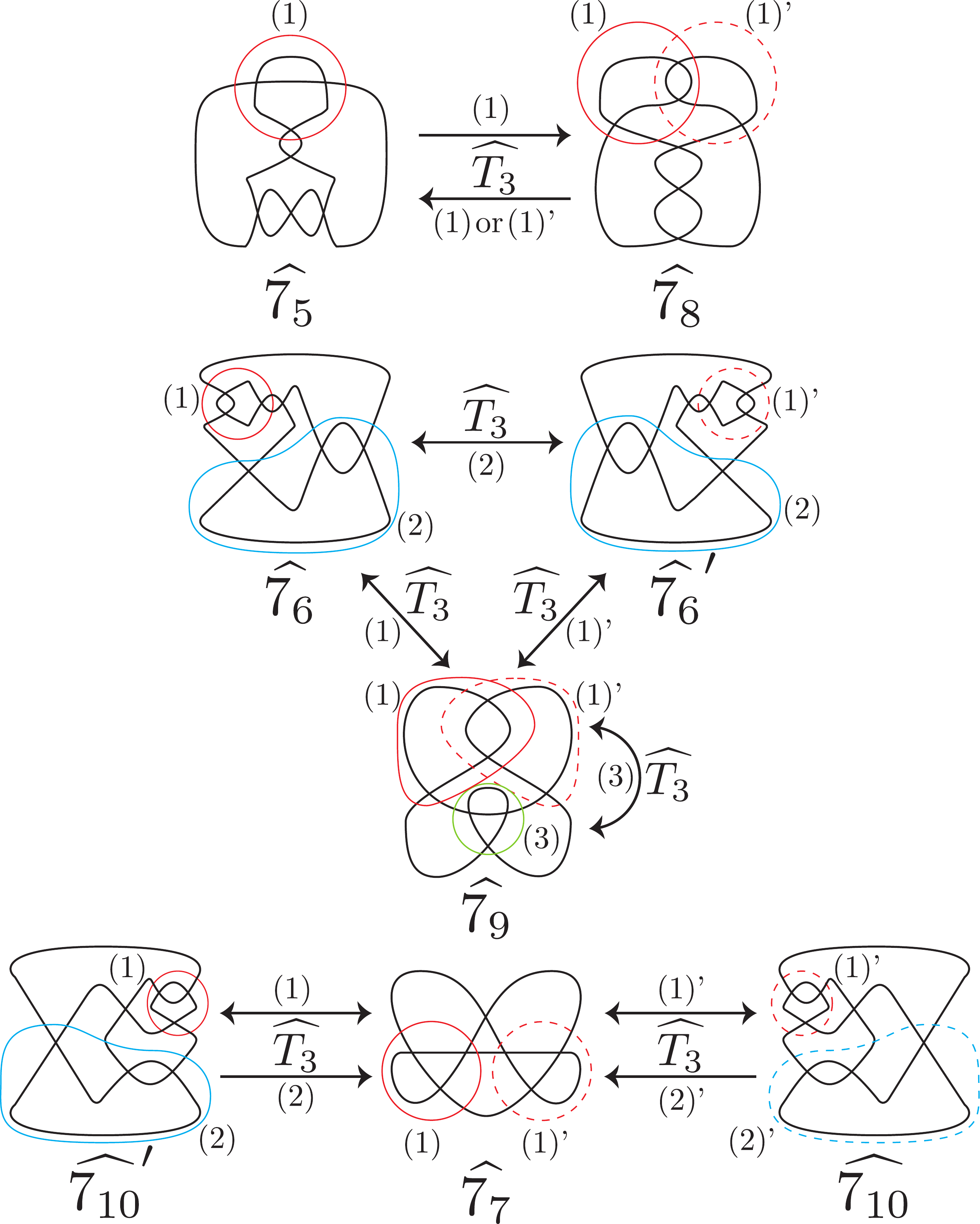}
\caption{$\widehat{7_{8}}$ is obtained from $\widehat{7_5}$ (top). $\widehat{7_9}$ is obtained from $\widehat{7_6}$ (center). $\widehat{7_{10}}$ is obtained from $\widehat{7_7}$ (bottom). }\label{7flype}
\end{figure}

Secondly, we seek a new knot projection $P''$  obtained from $P'$ ($=$ $\widehat{7_8}$, $\widehat{7_9}$, or $\widehat{7_{10}}$) by applying a flype of a crossing across the tangle $\widehat{T_{2}}$.    
However, there is no such $\widehat{T_{2}}$ (it is elementary to check every disk of type $\widehat{T_{2}}$ for $\widehat{7_8}$, $\widehat{7_9}$, or $\widehat{7_{10}}$).    
Thus, $\mathcal{P}_7$ $=$ $Alt_7$ $\cup$ $\{\widehat{7_8}$, $\widehat{7_9}, \widehat{7_{10}} \}$ that consists of knot projections up to mirror symmetry.    
\begin{remark}
For a number ($n$), we often use the symbol ($n$)', which is identified with ($n$) up to reflection on $S^2$ if necessary.  
\end{remark}
\subsection{\textbf Step~2c: Up to eight double points}\label{8_flype_sec}
Step~2c is the same process as Step~2b.  
A table $Alt_8$ is known as Figure~\ref{8}.  
\begin{figure}[h!]
\includegraphics[width=11cm]{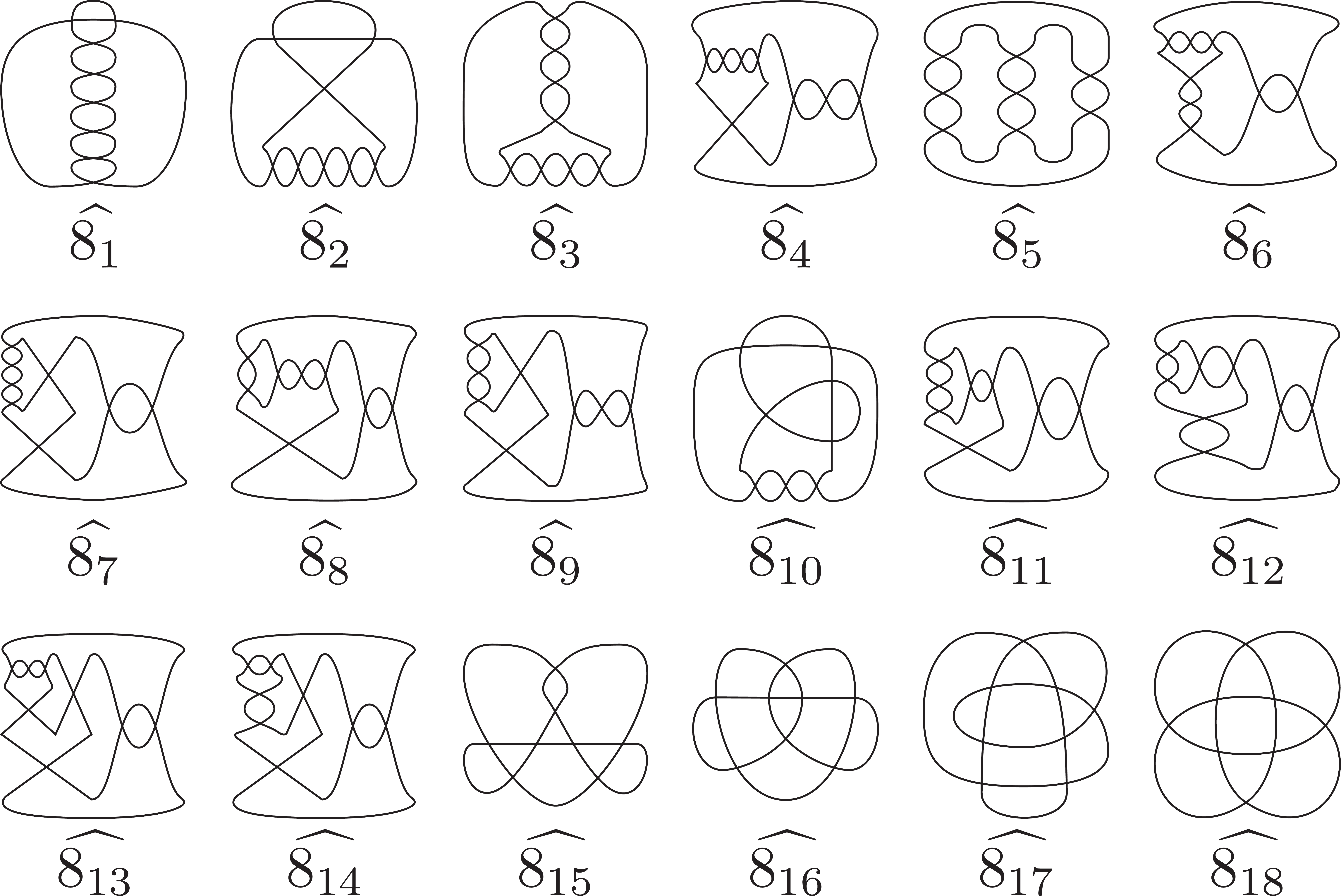}
\caption{$Alt_8$.}\label{8}
\end{figure}
By Lemma~\ref{lemma0} and Lemma~\ref{lemma5}, for a knot projection, if there exists a nontrivial flyping which is caused, there exists
$\widehat{T_{2}}$ such that $\widehat{T_3}$ $=1+\widehat{T_{2}}$ ($\widehat{T_{4c}}$ such that $\widehat{T_{4c}}$ $=1+\widehat{U_3}$, resp.), for the knot projection.  
Thus, for each $P \in Alt_8$, if a knot projection $P'$ is obtained from $P$ by applying a flype of a crossing across the tangle $\widehat{T_{2}}$ and $P' \neq P$, $P' \in \mathcal{P}_8$.  
There exist tangles, each of which corresponds to nontrivial flyping, as shown in Figures~\ref{8flype}--\ref{8flypec} (to  see these figures, see Figure~\ref{7flype1}, for example).  
Thus, $\widehat{8_{19}}$--$\widehat{8_{22}}$ and $\widehat{8_{24}}$--$\widehat{8_{27}} \in \mathcal{P}_8$.   

Secondly, we seek a new knot projection $P''$  obtained from the above $P'$ ($=$ $\widehat{8_{i}}$ ($19 \le i \le 27, i \neq 23$)) by applying a flype of a crossing across the tangle $\widehat{T_{2}}$ or $\widehat{U_{3}}$, we should add $P'' \in \mathcal{P}_8$.  
\begin{figure}[h!]
\includegraphics[width=12cm]{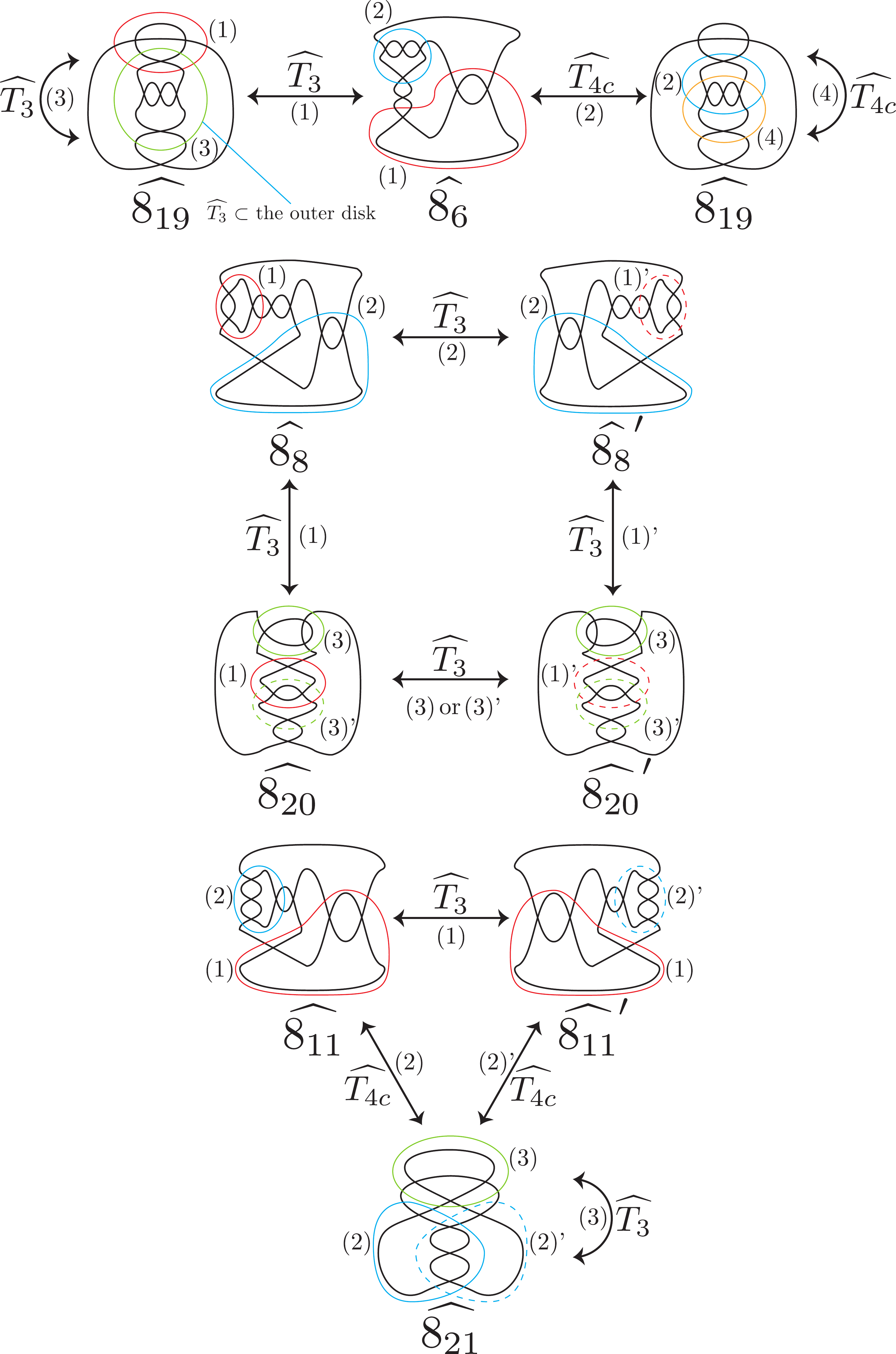}
\caption{The segment indicates a single flype of a crossing across the tangle $\widehat{T_{2}}$ or $\widehat{U_{3}}$.  Each dotted circle denotes one of the other choices of ambient disks.}\label{8flype}
\end{figure}
\begin{figure}[h!]
\includegraphics[width=12cm]{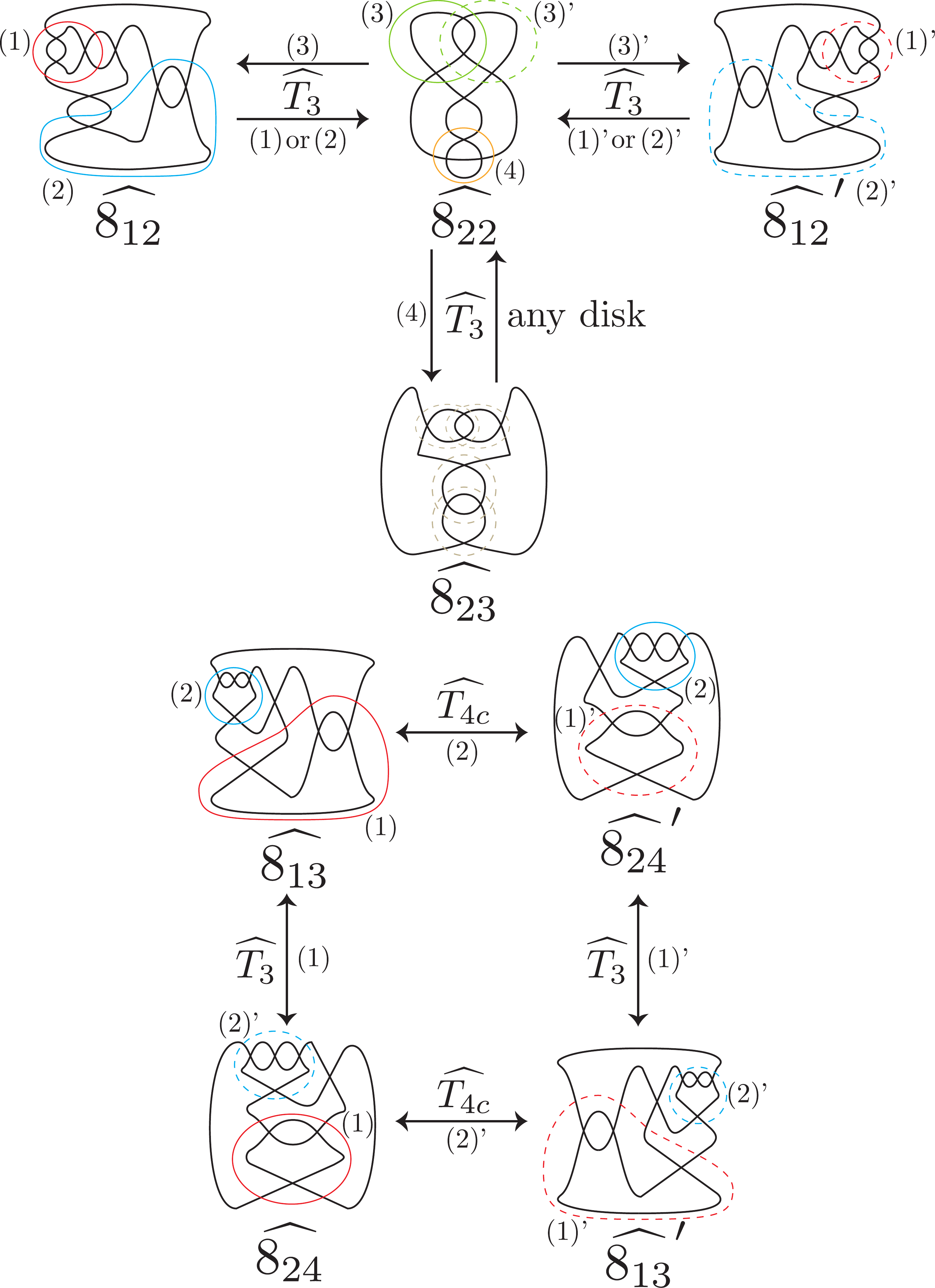}
\caption{The segment indicates a single flype of a crossing across the tangle $\widehat{T_{2}}$ or $\widehat{U_{3}}$.  Each dotted circle denotes one of the other choices of ambient disks.}\label{8flypeb}
\end{figure}
\begin{figure}[h!]
\includegraphics[width=12cm]{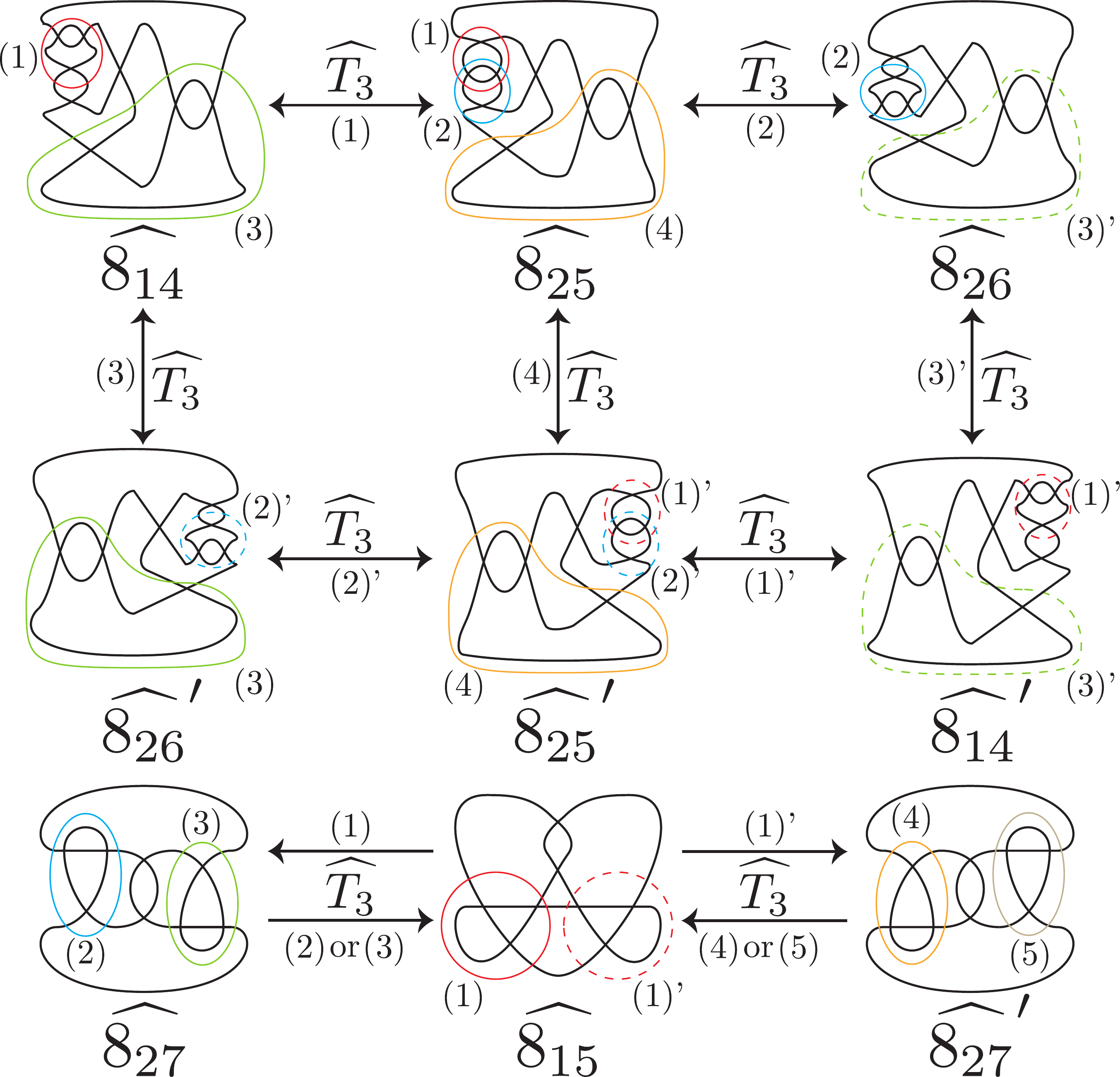}
\caption{The segment indicates a single flype of a crossing across the tangle $\widehat{T_{2}}$ or $\widehat{U_{3}}$.  Each dotted circle denotes one of the other choices of ambient disks.}\label{8flypec}
\end{figure}
Then, $\widehat{8_{23}} \in \mathcal{P}_8$. 
 
Thirdly, we seek a new knot projection $P'''$  obtained from $\widehat{8_{23}}$ by applying a flype of a crossing across the tangle $\widehat{T_{2}}$ or $\widehat{U_{3}}$.  However, there is no such flype for $\widehat{8_{23}}$.  Thus, $\mathcal{P}_8$ $=$ $Alt_8$ $\cup$ $\{ \widehat{8_i}~|~ 19 \le i \le 27 \}$.  Then, we have the complete list $\mathcal{P}_{\le 8}$ that consists of knot projections up to mirror symmetry.  

\subsection{\textbf Step~3: Assembling mirror images by using arrow diagrams of knot projections}\label{secA}
Recall the definitions of $\mathcal{P}_{\le n}$ and $\mathcal{P}'_{\le n}$ of Definition~\ref{def2}.  
In this section, we recall the definition of arrow diagrams
(Definition~\ref{dfn2_arrow}), which implies a map $\mathcal{P}_{\le n} \cup \mathcal{P}'_{\le n}$ to the set of arrow diagrams (Definition~\ref{dfn_cdp}).    
The map completely detects the difference between a knot projection and its mirror image (Lemma~\ref{prop1}).  
By applying it to $\mathcal{P}_{\le 8}$, we complete the proof of Theorem~\ref{mainthm}, i.e., we have $\mathcal{P}_{\le 8} \cup \mathcal{P}'_{\le 8}$.     
\begin{definition}[arrow diagram]\label{dfn2_arrow} 
An {\it{arrow diagram}} is a configuration of $n$ pair(s) of points up to ambient isotopy and reflection on a circle, where each pair of points consists of a starting point and an end point.    
Traditionally, two points of each pair are connected by a straight arc.  The straight arc is called a \emph{chord}.  Then an assignment of starting and end points on the boundary points of a straight arc is represented by an arrow on the chord from the starting point to the end point.     
\end{definition}
\begin{definition}[an arrow diagram of a knot projection $P$]\label{dfn_cdp}
Let $P$ be a knot projection. Then, there is a generic immersion $g: S^1 \to S^2$ such that $g(S^1)=P$.  
We define an arrow diagram of $P$  as follows (Figure~\ref{def1}). Let $l$ be the number of the double points of $P$.   
\begin{figure}[h!]
\includegraphics[width=8cm]{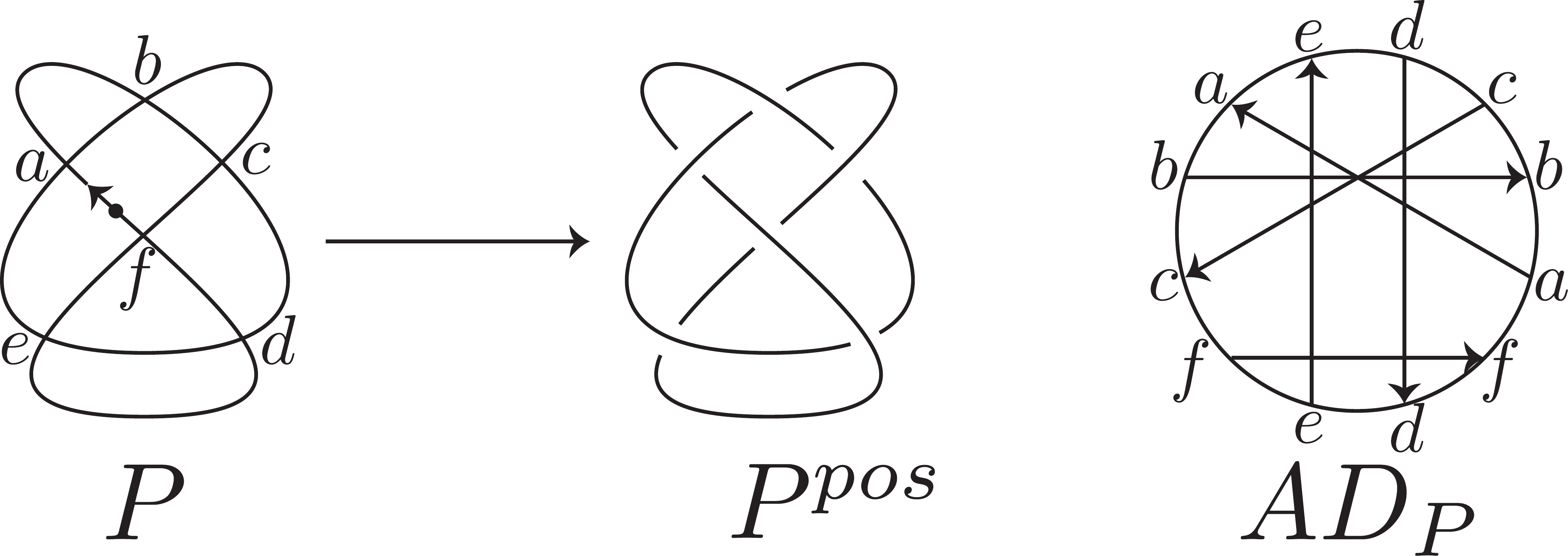}
\caption{An arrow diagram $AD_P$ of a knot projection $P$ via a knot diagram $P^{pos}$.}\label{def1}
\end{figure} 
\begin{figure}[htbp]
\includegraphics[width=6cm]{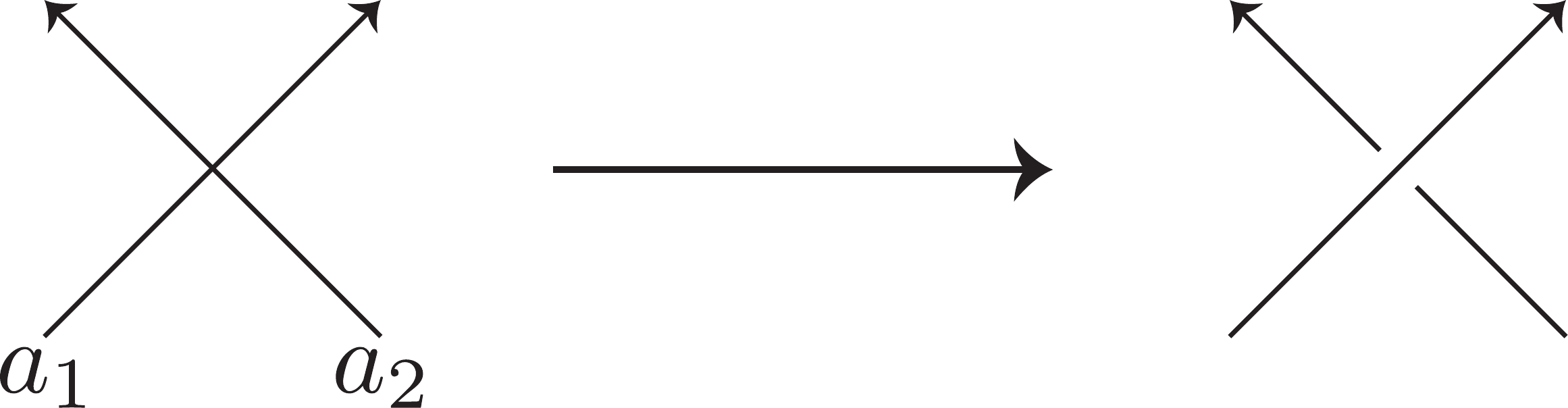}
\caption{A replacement of a double point with a crossing.}\label{positive}
\end{figure}
We fix a base point, which is not a double point on $P$. Then we choose an orientation of $P$.  After we start from the base point, we proceed along $P$ according to the orientation of $P$.    
Assign $1$ to the first double point which we encounter.  Then we assign $2$ to the next double point which we encounter provided it is not the first double point.  Suppose that we have already assigned $1$, $2, \dots, k$.  Then we assign ${k+1}$ to the next double point which we encounter provided it has not been assigned yet.    
Following the same procedure, we finally label the double points of $P$.    
Here, note that $g^{-1}({\text{double point assigned $i$}})$ consists of two points on $S^1$.  
Now we focus on the double point corresponding to the two points.  Suppose that we regard the double point as the left of Figure~\ref{positive}.  
The left of Figure~\ref{positive} consists of two oriented paths, e.g., $a_1$ and $a_2$, where $a_1$ traverses $a_2$ from the left side and $a_2$ traverses $a_1$ from the right side.   Then, we
 connect the label $a_1$ and $a_2$ along the circles with an arrow pointing from $a_1$ to $a_2$, see Figure~\ref{def1}.  
The arrow diagram represented by $g^{-1}({\text{double point assigned $1$}}),$ $g^{-1}({\text{double point assigned $2$}}),$ $\dots,$ $g^{-1}({\text{double point assigned $l$}})$ on $S^1$ is denoted by $AD_P$ and is called {\it{an arrow diagram of the knot projection}} $P$.  Denote by $P^{pos}$ a knot diagram obtained by the replacement, as shown in Figure~\ref{positive}, of each double point (e.g., the center of Figure~\ref{def1}).    
\end{definition}
Note that $AD_P$ does not depend on the base point and thus, it is well-defined up to orientations of $P$.  
Thus, we have Proposition~\ref{prop1}.
\begin{proposition}\label{prop1}
A knot projection $P$ is equivalent to its mirror image $P'$ up to ambient isotopy on a $2$-sphere if and only if $AD_{P}$ and $AD_{P'}$ are equivalent up to ambient isotopy and reflection on a plane.  

Further, $AD_{P'}$ is obtained from $AD_P$ by replacing the orientation of each arrow with the inverse orientation.    
\end{proposition}
By using Proposition~\ref{prop1}, $P$ and its mirror image $P'$ are different if $AD_{P}$ is not $AD_{P'}$ (here, note that $AD_P$ is defined up to ambient isotopy and reflection).      
The table (Tables~\ref{arrow_table1}--\ref{arrow_table3}) of the arrow diagrams $AD_P$ for each $P$ in Tables~\ref{table1} and \ref{table4}.    
\begin{table}[htbp]
\caption{Arrow diagrams corresponding to $\mathcal{P}_{\le 8}$ ($\widehat{1_1}$ to $\widehat{8_6}$).}\label{arrow_table1}
\includegraphics[width=10cm]{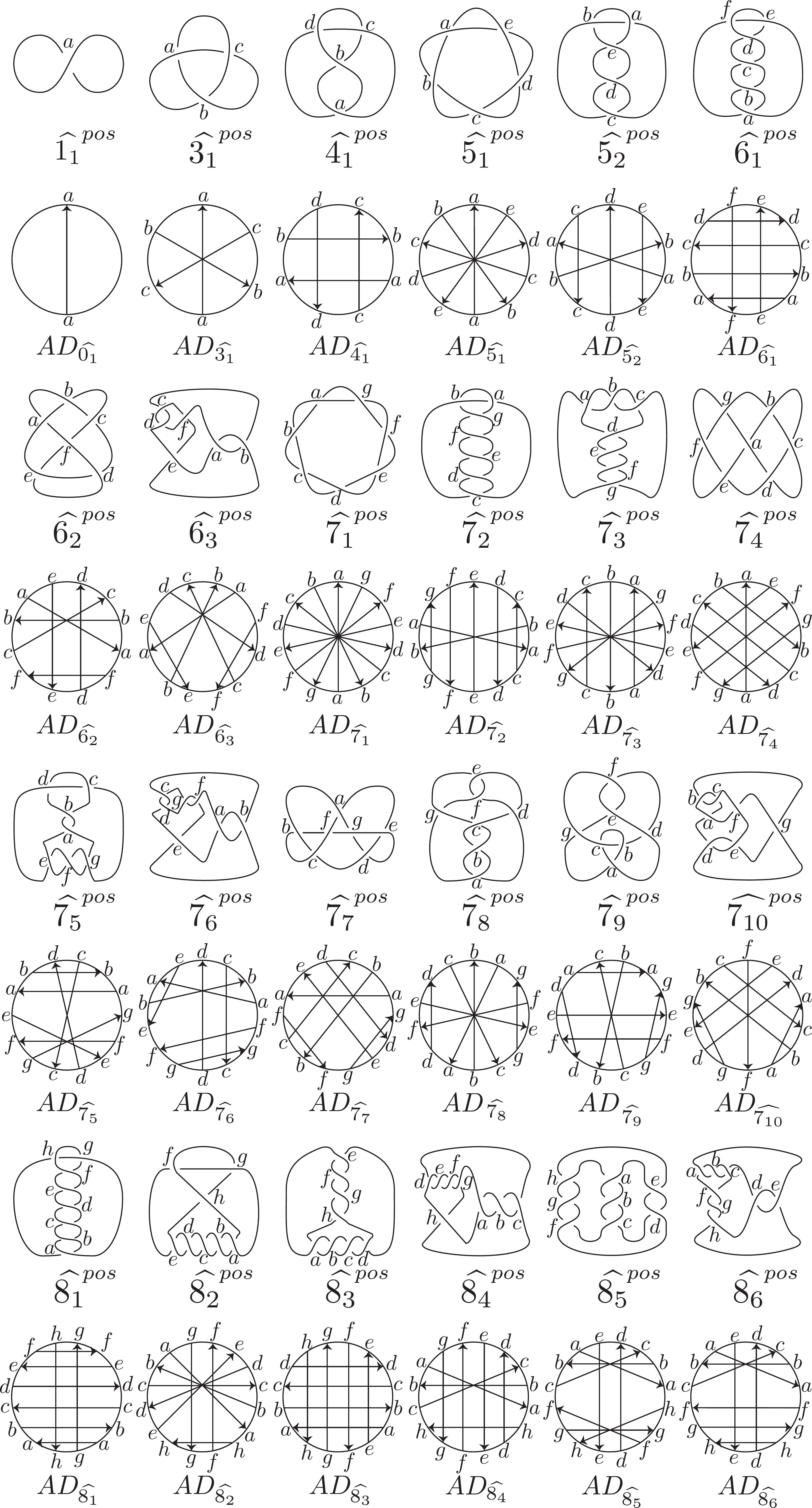}
\end{table} 
\begin{table}[htbp]
\caption{Arrow diagrams corresponding to $\mathcal{P}_{\le 8}$ ($\widehat{8_7}$ to $\widehat{8_{27}}$). }\label{arrow_table2}
\includegraphics[width=10cm]{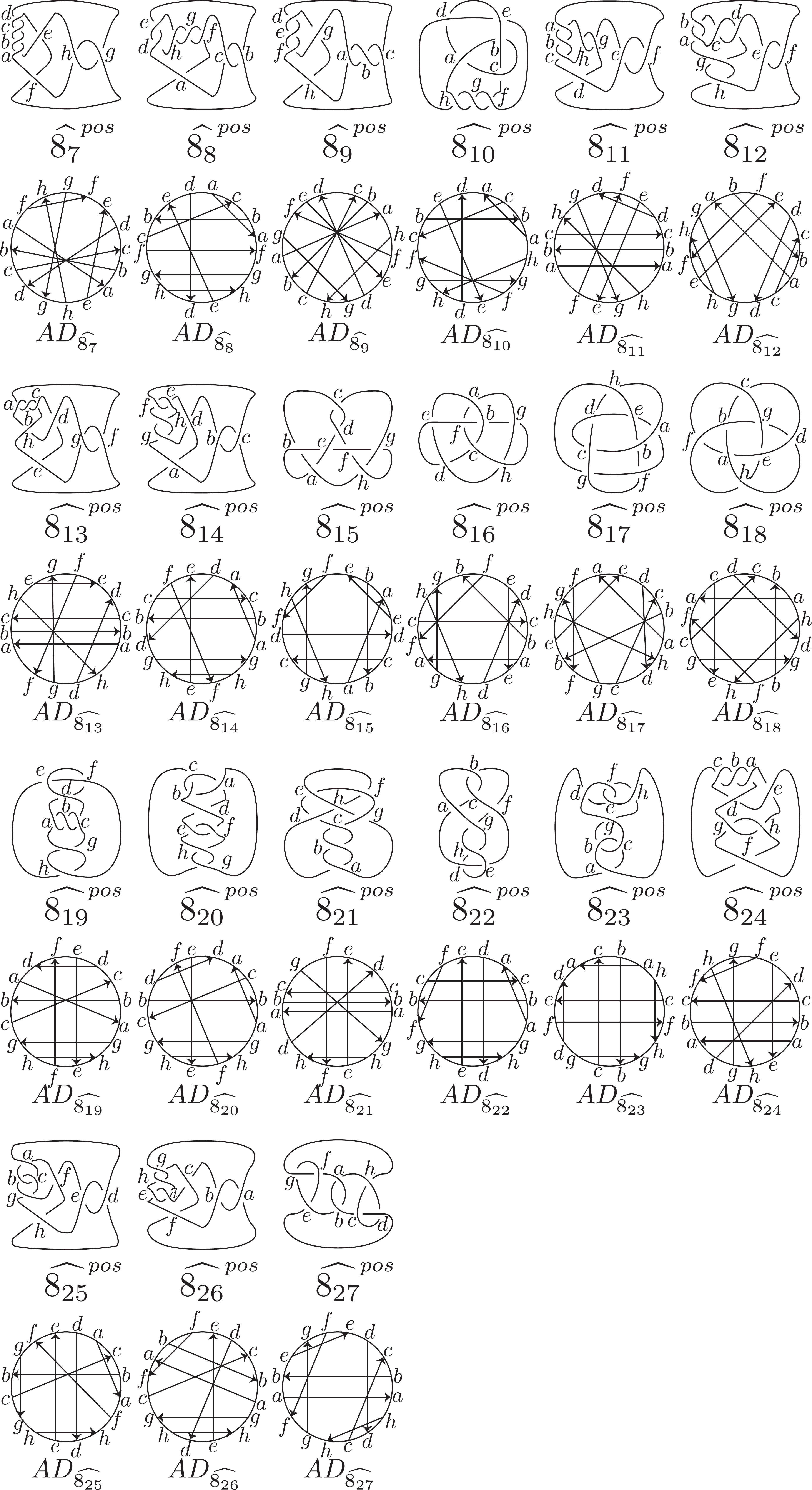}
\end{table} 
\begin{table}[htbp]
\caption{Arrow diagrams corresponding to $\mathcal{P}'_{\le 8} \setminus \mathcal{P}_{\le 8}$.}\label{arrow_table3}
\includegraphics[width=10cm]{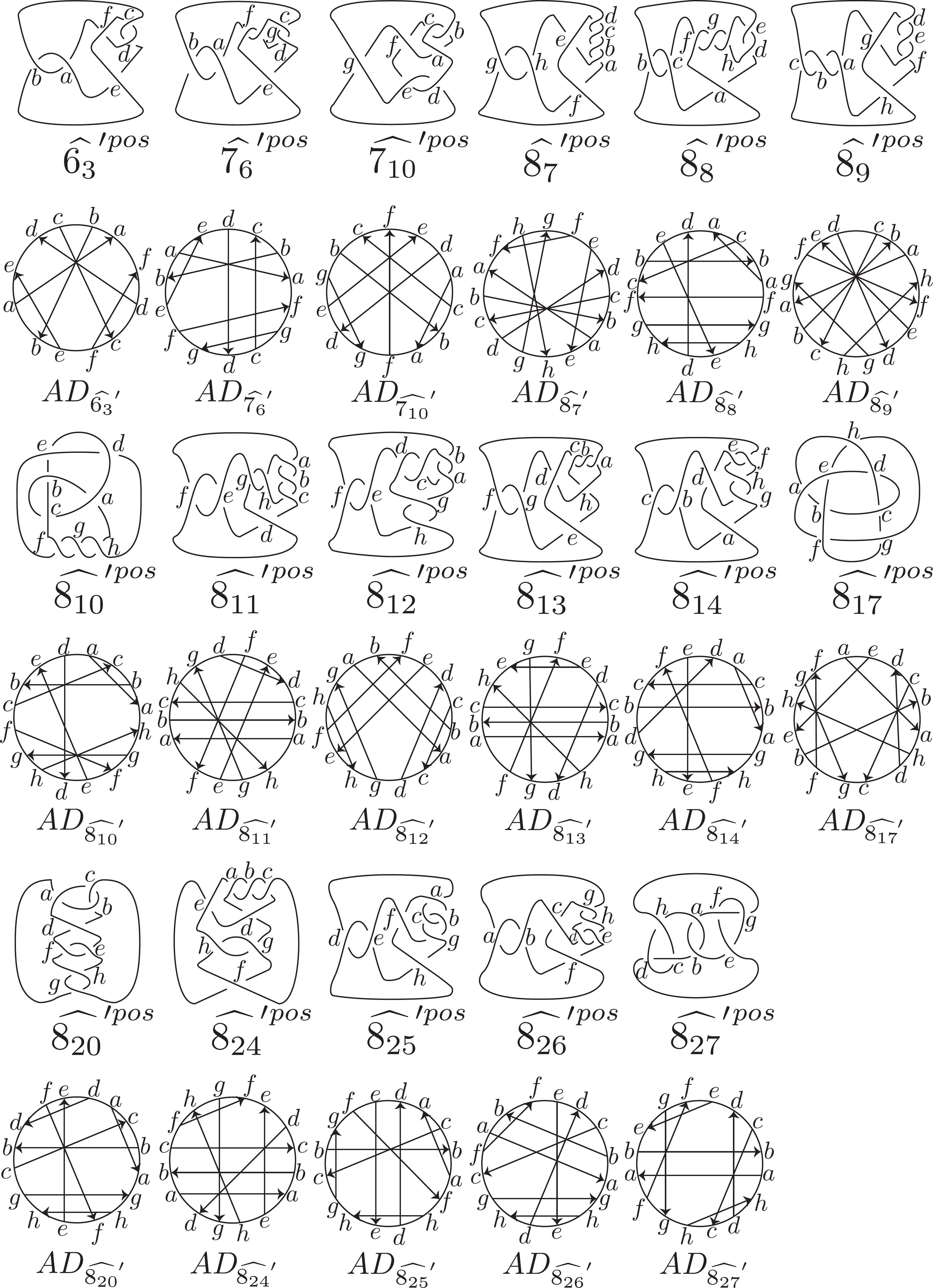}
\end{table}

\section*{\textbf Acknowledgements}
The authors would like to thank the referee for  useful comments.    
This work is partially supported by Sumitomo Foundation (Grant for Basic Science Research Projects, Project number: 160556).  N.~Ito was a project researcher of Grant-in-Aid for Scientific Research~(S) 24224002.

\end{document}